\documentclass[
reqno, 
12pt]{amsart}
\usepackage{amssymb,amsmath, amsthm,latexsym, mathrsfs}

\usepackage{a4wide}

\usepackage{hyperref}

\usepackage[all]{xy}

\usepackage{tikz}

\theoremstyle{plain}

\newtheorem{lemma}{Lemma}[section]
\newtheorem{theorem}[lemma]{Theorem}
\newtheorem{proposition}[lemma]{Proposition}

\newtheorem{remark}[lemma]{Remark}

\newtheorem{definition}[lemma]{Definition}

\parskip=\bigskipamount

\def\Z{\mathbb Z}
\def\R{\mathbb R}

\def\d{\delta}

\def\t{\times}
\def\o{\otimes}

\def\ra{\rightarrow}
\def\la{\leftarrow}

\def\a{\alpha}

\def\z{\mathfrak z}

\def\D{\Delta}
\def\G{\Gamma}

\def\s{\sigma}

\def\SR{ \mathcal S(\mathbb R)}
\def\SC{ \mathcal S(\mathcal C)}

\def\DS{\mathcal D}

\def\C{\mathcal }

\title[Irreducibility criterion for  induced  representations]
{Irreducibility criterion for representations \\ induced   by essentially unitary ones \\ (case of
 non-archimedean  $GL(n,\C A)$)}

\mark{Marko Tadi\'c}

\author{Marko Tadi\'c}

\address{Department of Mathematics, University of Zagreb
\\
Bijeni\v{c}ka 30, 10000 Zagreb,
 Croatia\\
Email: \tt tadic{\char'100}math.hr}


\thanks{The   
author was partly supported by 
Croatian Ministry of Science, Education and Sports grant
{\#}037-0372794-2804.}
\date{\today}

\begin{document}

\begin{abstract} 
Let $\C A$ be a finite dimensional central division algebra over a  local  non-archimedean field $F$. 
Fix any parabolic subgroup $P$ of $GL(n,\C A)$ and  a Levi factor $M$ of $P$. Let $\pi$ be an irreducible unitary representation of $M$ and $\varphi$ a (not necessarily unitary) character of $M$. We give an explicit necessary and sufficient condition for the parabolically induced representation
$$
\text{Ind}_P^{GL(n,\C A)}(\varphi\pi)
$$
to be irreducible.
\end{abstract}

\maketitle

\setcounter{tocdepth}{1}

%\begin{centerline}
%{--------- \ \ \ \  Preliminary version \ \ \ \  ---------}
%\end{centerline}
%

%\tableofcontents

\section{Introduction}\label{intro}

Let $F$ be a local non-archimedean field and let $\mathcal A$ be a finite dimensional central division algebra of rank $d_\C A$ over $F$. Put
$$
G_p=
GL(p,\C A).
$$
 For an irreducible essentially square integrable representation $\d$ of $G_p$, denote by $s_\d$ the smallest positive real number such that 
 $$
 \text{Ind}^{G_{2p}}(\d\o|\det |_F^{s(\d)}\d)
 $$
 reduces. Then $s_\d\in\Z$ and $s_\d|d_{\C A}$. Let
 $$
 \nu_\d:=|\det |_F^{s(\d)}.
 $$
For $A,B\in\mathbb Z$, $A\leq B$,  
the set $\{x\in\Z; A\leq x\leq B\}$ is called a $\Z$-segment. It is denoted by
$$
[A,B]_\Z.
$$
Let $\rho$ be an irreducible cuspidal representation of $G_p$.
Then we call
$$
[A,B]^{(\rho)}:=\{\nu_\rho^i\,\rho\, ;\ i\in [A,B]_\Z\}
$$ 
a segment in cuspidal representations.

Let  $\D=
[A,B]^{(\rho)}$. Consider the representation
$$
\text{Ind}^{G_{(n+1)p}}(\nu_\d^B\rho\o\nu_\d^{B-1}\rho\o\dots\o\nu_\rho^A\rho),
$$
parabolically induced from the appropriate  parabolic subgroup containing regular upper triangular matrices (see the second section).
Then the above representation has a unique irreducible subrepresentation, denoted by
$$
\d(\D).
$$
Further, $\d(\D) $ is essentially square integrable, and we get all the irreducible essentially square integrable representations in this way. Irreducible essentially square integrable representations are basic building blocks in the classification of non-unitary duals of general linear groups over $\C A$ via Langlands classification (see \cite{Z} and \cite{T-div-a} among others).

One of the first  cornerstones of the  representation theory of general linear groups over $\C A$ is the reducibility criterion for 
$$
\text{Ind}^{G_{n_1+n_2}}(\d_1\o\d_2)
$$
where $\d_i$ are irreducible essentially square integrable representations of $G_{n_i}$. We can have reducibility  only if we can write $\d_i=\d([A_i,B_i]^{(\rho)})$, $i=1,2$, for some integers $A_i,B_i$, and an  irreducible cuspidal representation $\rho$. Then
\begin{equation}
\label{seg-red}
\text{Ind}^{G_{n_1+n_2}}(\d([A_1,B_1]^{(\rho)})\o\d([A_2,B_2]^{(\rho)}))
\end{equation}
reduces if and only if holds the following
\begin{enumerate}
\item $[A_1,B_1] \cup [A_2,B_2]$ is a $\Z$-segment;

\item $A_1<A_2$ and $B_1<B_2$, or 
conversely\footnote{i.e. $A_2<A_1$ and $B_2<B_1$}.
\end{enumerate}

Let $\d=\d([A,B]^{(\rho)})$ be an irreducible essentially square integrable representation of $G_p$ and let $n$ be a positive integer. Denote
$$
C=A+n-1,\qquad D=B+n-1.
$$
Then the  representation 
 $$
 \text{Ind}^{G_{np}}(\d([\nu_\d^{C}\rho,\nu_\d^{D}\rho])\o \d([\nu_\d^{C-1}\rho,\nu_\d^{D-1}\rho])\o\dots \o\d([\nu_\d^{A}\rho,\nu_\d^{B}\rho])).
 $$
 has  a unique irreducible quotient\footnote{Obviously, for $n=1$ this is $\d([A,B]^{(\rho)})$.}, denoted by
 \begin{equation}
 \label{speh-def}
 u_{ess}
\left(
\begin{matrix}
A  \hskip3mm & B  \hskip3mm
\\
    \hskip3mm C&    \hskip3mm D
\end{matrix}
 \right)^{(\rho)}.
 \end{equation}
This representation is essentially unitarizable
(i.e., it becomes unitarizable after twist by a character; see \cite{BR-Tad}),  and called essentially Speh representation. Representations \eqref{speh-def} are basic building blocks in the classification of unitary duals of general linear groups over $\C A$  (see \cite{T-AENS}, \cite{Se} and \cite{BHLS} among others). Irreducible unitary representations are fully induced by a tensor product of essentially Speh representations.

Now we shall present a simple and natural  generalization of the above criterion for reducibility of \eqref{seg-red} to the case of essentially Speh representations.
First, we define
$$
\left(
\begin{matrix}
A_1 \hskip3mm & B_1  \hskip3mm
\\
    \hskip3mm C_1&    \hskip3mm D_1
\end{matrix}
 \right)
 <_{strong}
 \left(
\begin{matrix}
A_2  \hskip3mm & B_2  \hskip3mm
\\
    \hskip3mm C_2&    \hskip3mm D_2
\end{matrix}
 \right)
\iff
A_1<A_2,\quad B_1<B_2,\quad C_1<C_2,\quad D_1<D_2.
$$

\begin{theorem} Let $\pi_1$ and $\pi_2$ be   essentially Speh representations of $G_{p_1}$ and $G_{p_2}$ respectively. If the representation
\begin{equation}
\label{r}
\text{Ind}^{\; G_{p_1+p_2}}(\pi_1\o \pi_2)
\end{equation}
   reduces, then we can find an irreducible cuspidal representation $\rho$ and  $A_i,B_i,C_i,D_i\in\Z$, such that
  $$
  \pi_i
  =u_{ess}
\left(
\begin{matrix}
A_i  \hskip3mm & B_i  \hskip3mm
\\
    \hskip3mm C_i&    \hskip3mm D_i
\end{matrix}
 \right)^{(\rho)}
  $$
  for $i=1,2$.
  Now 
  $\text{Ind}^{\; G_{p_1+p_2}}(\pi_1\o \pi_2)$
   reduces if and only if
      \begin{enumerate}
\item
 $
[ A_1, D_1]_\Z\cup [ A_2, D_2]_\Z
 $
  is a $\Z$-segment; 
  \item
$$
\left(
\begin{matrix}
A_1  \hskip3mm & B_1  \hskip3mm
\\
    \hskip3mm C_1&    \hskip3mm D_1
\end{matrix}
 \right)
 <_{strong}
 \left(
\begin{matrix}
A_2  \hskip3mm & B_2  \hskip3mm
\\
    \hskip3mm C_2&    \hskip3mm D_2
\end{matrix}
 \right) 
 \qquad
 \text{or conversly.}
 $$
\end{enumerate}
\end{theorem}

The proof of the above criterion\footnote{Which is an obvious generalization of the criterion for irreducibility of \eqref{seg-red}, since there $A_i=C_i$ and $B_i=D_i$}
 is rather  elementary, and we shall comment it very briefly below. It is based on a simple and natural irreducibility criterion of I. Badulescu, E. Lapid and A. M\'inguez obtained in \cite{BLM} (recalled in the subsection 3.4 of this paper), and simple  combinatorial algorithm obtained by C. M\oe glin and J.-L. Waldspurger in \cite{MoeW-alg}  describing the Zelevinsky involution $a\mapsto a^t$ on the multi sets of segments of cuspidal representations (recalled in the subsection 2.6 of this paper).

All the reducibility's which show up in the above theorem are direct consequence of the fact that in this case $(a_1+a_2)^t\ne a_1^t+a_2^t$ for  multi sets of segments of cuspidal representations that parameterize corresponding essentially Speh representations (see Proposition 5.1).

The irreducibility which show up in the above theorem is obtained in two ways. One is direct application of criterion of I. Badulescu, E. Lapid and A. M\'inguez. If  this criterion does not imply irreducibility, then we show that $(a_1+a_2)^t= a_1^t+a_2^t$ for  multi sets of segments of cuspidal representations that parameterize corresponding essentially Speh representations (Lemma 6.1), and that cannot happen $b<a_1+a_2$ and $b^t<(a_1+a_2)^t$ for any multi set $b$ of segments of cuspidal representations (see the proof of Proposition 6.3)\footnote{Here $<$ is a natural ordering on multi sets of segments of cuspidal representations (see 2.3).}. This easily implies  irreducibility (see 3.2). This method of proving irreducibility was used by I. Badulescu in \cite{Ba-Sp}.

The above theorem easily implies the following

\begin{theorem} Suppose that we have essentially Speh representations $\pi_1,\dots,\pi_k$. Then 
$$
\pi_1\t\dots\t\pi_k
$$
is irreducible if and only if  the representations
$$
\pi_i\t\pi_j
$$
are irreducible for all $1\leq i<j\leq l.$
\qed
\end{theorem}

The problem of reducibility that we study in this paper was studied by C. M\oe glin and J.-L. Waldspurger in the case when $\C A$ is a field. They have proved one implication stated in Theorem 1.1    (see Lemma I.6.3 of \cite{MoeW-GL}; this is harder implication). 
The proof of irreducibility in  \cite{MoeW-GL}  is based on analytic properties of standard integral intertwining operators, normalized 
 by $L$-functions and $\epsilon$-factors.
 The paper  \cite{MoeW-GL} essentially contains (for the field case) a proof Theorem 1.2 using machinery of intertwining operators (see I.8 and  Proposition I.9 of \cite{MoeW-GL}). In our proof of Theorem 1.2, it follows easily and completely elementary from Theorem 1.1.

Further,  a specialization of  Theorem 1 of B. Leclerc, M. Nazarov and J.-Y. Thibon from \cite{LNT} (which addresses Hecke algebra representations) to the case of unramified representations of
general linear groups over a non-archimedean local field $F$, implies our result for unramified essentially Speh representations (their unramified result is more general). 
The theory of types for general linear groups over division algebras, developed in \cite{Se1} - \cite{SeSt4}, together with the theory of covers from \cite{BK}, should relatively easily imply  that Theorem \ref{LNT-Th} extends in a natural way also to the general case, but we have not checked all details of the implication (one can find in \cite{Se} and \cite{BHLS}  such type of applications of \cite{Se1} - \cite{SeSt4} and \cite{BK}). 
This way of proving the irreducibility criterion is  technically very complicated (already in the unramified case, where it uses \cite{L}). Since the claim of our main result does not include types, it is interesting to have a proof of it which does not use types (in particular, if it is relatively simple).

Discussions with I. Badulescu, E. Lapid and C. Moeglin were helpful in the course of preparation of this paper.
 C. Jantzen's numerous corrections helped us a lot to  improve the style of the paper.
  A. M\'inguez has explained us  how to  get alternative proof based on Jacquet modules of the main result of section 6  (by simple use of Lemma 1.2 from \cite{BLM}; see Remark \ref{rem-Min} of this paper for a few more details).  
  We are thankful to all them.

The content of the paper is as follows. In the second section we recall  notation for general linear groups that we use in this paper. The third section recalls  some very simple criteria for reducibility or irreducibility of parabolically induced representations. In the fourth section we consider  relations between segments defining essentially Speh representations. In the fifth section we prove the reducibility criterion for two essentially Speh representations  in the case when the underlying sets of cuspidal supports are linked, while in the sixth section we prove the criterion when they are not linked. The seventh section gives two formulations of the criterion  which we have proved in the previous sections. Here we also address the case of several  essentially Speh representations.   The eight section clarifies the relation with the work of C. M\oe glin and J.-L. Waldspurger, while the last section  clarifies the relation with the work of B. Leclerc, M. Nazarov and J.-Y. Thibon.

\section{Notation}\label{notation}

We  recall some notation for general linear groups in the non-archimedean case, following mainly \cite{T-div-a}, \cite{Ro} and \cite{Z}. 

\subsection{$\Z$-segments in $\R$} By a $\mathbb Z$-segment in $\mathbb R$, we mean a set of the form
$$
\{x,x+1,\dots,x+n\},
$$
where $x\in\mathbb R$ and $n\in\Z_{\geq 0}$. We  denote the above set by
$$
[x,x+n]_\Z,
$$
or later on, simply by
$$
[x,x+n]
$$
 to shorten notation (this will not cause confusion since we shall not deal with intervals  of real numbers in this paper).
Then $x$ is called the beginning of $\D$, and denoted by $b(\D)$, and $x+n$ is called the end of $\D$, and denoted by $e(\D)$.
We  denote the set of all $\Z$-segments in $\R$ by $\mathcal S(\R)$. For $[x,y]_\Z \in \mathcal S(\R)$, let
$$
[x,y]_\Z^-=[x,y-1]_\Z,
$$
$$
^-[x,y]_\Z=[x+1,y]_\Z
$$
if $x<y$. Otherwise, we take $[x,x]_\Z^-=\, ^-[x,x]_\Z=\emptyset$. 

For $n\in\Z_{>0}$, let
$$
\D[n]=[-(n-1)/2,(n-1)/2]_\Z.
$$

Segments $\D_1,\D_2 \in \mathcal S(\mathbb R)$ are called linked if $\D_1\cup\D_2\in\SR$ 
and $\D_1\cup\D_2\not\in\{\D_1,\D_2\}$. If the segments $\D_1$ and $\D_2$ are linked and if $\D_1$ and $\D_1\cup\D_2$ have the same beginnings, we say that 
$\D_1$ precedes $\D_2$, and write
$$
\D_1\rightarrow \D_2.
$$ 
For $\D\in\SR$ and $x\in\R$, let
$$
\D_x:=\{x+y;y\in\D\}\in\SR.
$$

For a set $X$, the set of all finite multisets in $X$ is denoted by $M(X)$ (we can view each multiset as a functions $X\rightarrow \Z_{\geq0}$ with finite support; note that  finite subsets correspond to all functions $X\rightarrow \{0,1\}$ with finite support).
Elements of $M(X)$ are denoted by $(x_1,\dots,x_n)$ (repetitions of elements can occur, and the  multiset does not change  if we permute $x_i$'s). The number 
$$
n
$$
 is called the cardinality of $(x_1,\dots,x_n)$. We  call 
$$
\{x_1,\dots,x_n\}
$$
the underlying set of $(x_1,\dots,x_n)$.

The set $M(X)$ has a natural structure of a commutative associative semi group with zero. The operation is denoted additively: 
$$
(x_1,\dots,x_n)+(y_1,\dots,y_m)=(x_1,\dots,x_n,y_1,\dots,y_m).
$$

Take positive integers $n$ and $d$. Let
\begin{equation}
\label{and}
a(n,d)=(\D[d]_{-\frac{n-1}2},\D[d]_{-\frac{n-1}2+1}, \dots,\D[d]_{\frac{n-1}2}) \in M(\SR).
\end{equation}

\subsection{Groups and representations}

Let $F$ be a non-archimedean locally compact non-discrete field and $|\ \ |_F$ its modulus character. 
Fix
 a finite dimensional central   division algebra $\mathcal A$ over $F$ of rank $d_{\mathcal A}$.
Denote by $\text{Mat} (n\t n,\mathcal A)$  the algebra of all $n\t n$ matrices with
entries in $\mathcal A$. Then $GL(n,\mathcal A)$ is the group of invertible matrices  with the
natural topology. The commutator subgroup is denoted by $SL(n,\mathcal A)$. 
Denote by 
$$
\text{det}: GL(n,\mathcal A) \ra GL(1,\mathcal A)/SL(1,\mathcal A)
$$
the determinant homomorphism, as defined  by J. Dieudonn\'e (for $n=1$ this
is just the quotient map). The kernel is $SL(n,\mathcal A)$.
Denote the reduced norm of $\text{Mat} (n\t n,\mathcal A)$ 
by r.n.$_{_{\text{Mat} (n\t n,\mathcal A)/F}}$.
We  identify characters of $GL(n,\mathcal A)$ with characters of $F^\times$
using r.n.$_{_{\text{Mat} (n\t n,\mathcal A)/F}}$. 
 Let
$$
\nu=|\text{r.n.}_{_{\text{Mat} (n\t n,\mathcal A)/F}}|_F: GL(n,\mathcal A) \ra \mathbb R^\times.
$$
Denote by
$$
G_n
$$
the general linear groups $GL(n,\mathcal A)
$ for $n\geq 0$ (we take $G_0$  to be the trivial group; we consider it formally  as the group of $0\times 0$ matrices).
 The category of all smooth representations of $G_n$ is denoted by Alg($G_n)$. The set of all equivalence classes of irreducible smooth representations of $G_n$ is denoted by $\tilde G_n$. The subset of unitarizable classes in $\tilde G_n$ is denoted by $\hat G_n$.
The Grothendieck group of the category of all smooth representations of $G_n$ of finite length is denoted by $R_n$. It is a free $\mathbb Z$-module with basis $\tilde G_n$.
The set of all finite sums in $R_n$ of elements of the basis $\tilde G_n$ is denoted by $(R_n)_+$. Set
$$
\aligned
&\text{Irr}=\cup_{n\in \mathbb Z_{\geq0}} \tilde G_n,
\\
&\text{Irr}^u=\cup_{n\in \mathbb Z_{\geq0}} \hat G_n,
\\
&R=\oplus_{n\in \mathbb Z_{\geq0}} R_n,
\\
&R_+=\sum_{n\in \mathbb Z_{\geq0}} (R_n)_+.
\endaligned
$$
The ordering on $R$ is defined by $r_1\leq r_2 \iff r_2-r_1\in R_+$.

The set of cuspidal classes in $\tilde G_n$ is denoted by $\mathcal C(G_n)$. Denote
$$
\aligned
&\mathcal C=\cup_{n\in \mathbb Z_{\geq1}} \mathcal C(G_n),
\\
&\mathcal C^u=\mathcal C\cap \text{Irr}^u.
\endaligned
$$

 Set  
$$
M_{(n_1,n_2)}:=\left\{
\left[
\begin{matrix}
g_1 & *
\\
0 & g_2
\end{matrix}
\right]
;
g_i\in G_{i}
\right\} \subseteq G_{n_1+n_2}.
$$
Let $\sigma_1$ and $\sigma_2$ be smooth representations of $G_{n_1}$ and $G_{n_2}$, respectively.
Consider $\s_1\o\s_2$ as a representation of $M_{(n_1,n_2)}$:
$$
\left[
\begin{matrix}
g_1 & *
\\
0 & g_2
\end{matrix}
\right]
\mapsto
\s_1(g_1)\o\s_2(g_2).
$$
Denote by 
$$
\s_1\t\s_2
$$
the representation of $G_{n_1+n_2}$ parabolically induced by $\s_1\o\s_2$ from $M_{(n_1,n_2)}$ (the induction that we consider here is smooth and normalized). Then for three representations, we have
\begin{equation}
\label{asso}
(\s_1\t\s_2)\t\s_3\cong \s_1\t(\s_2\t\s_3).
\end{equation}
Since the induction functor is exact, we can lift it in a natural way to a $\mathbb Z$-bilinear mapping $\t:R_{n_1}\t R_{n_2}\rightarrow R_{n_1+n_2}$, and further to $\t : R\t R\rightarrow R$. In this way $R$ becomes graded commutative ring.

For $\rho\in \mathcal C$ denote by 
$$
s_\rho
$$
the minimal non-negative number such that 
$
\rho\t \nu^{s_\rho}\rho
$
 reduces. Then $s_\rho\in \Z_{\geq 1}$, and it divides $d_{\mathcal A}$ (it can be described in terms of Jacquet-Langlands correspondence established in \cite{DKV}). Put
$$
\nu_\rho:=\nu^{s_\rho}.
$$

\subsection{Segments in cuspidal representations $\mathcal C$} Let $\D\in \SR$ and $\rho\in\mathcal C$. Set
$$
\D^{(\rho)}:=\{\nu_\rho^x\rho;x\in\D\}.
$$
The set $\D^{(\rho)}$ is called a segment in $\mathcal C$. Once  we fix $\rho$, then we call elements in $\D$ the exponents of elements in $\D^{(\rho)}$.  Then,  when we work with $\D^{(\rho)}$, we often drop the superscript $(\rho)$, and instead  of $\D^{(\rho)}$ and its elements, we refer   simply  to $\D$ and its elements.

The set of all segments in $\mathcal C$ is denoted by $\SC$.
We take $\emptyset^{(\rho)}=\emptyset$.
For $\D^{(\rho)}\in\SC$, where $\D\in\SR$ and $\rho\in\mathcal C$, we define
$$
\aligned
&(\D^{(\rho)})^-:= (\D^-)^{(\rho)},
\\
&^-(\D^{(\rho)}):=(\,^-\D)^{(\rho)}.
\endaligned
$$
For two segments $\G_1,\G_2\in\SC$, we say that they are linked if there exist linked segments $\D_1,\D_2$ in $\SR$ and $\rho\in\mathcal C$ such that
$$
\G_i=\D_i^{(\rho)},\quad i=1,2.
$$
In that case we say that $\G_1$ precedes $\G_2$ if $\D_1$ precedes $\D_2$, and  we then write
$$
\G_1\rightarrow \G_2.
$$

For $a=(\D_1,\dots,\D_n)\in M(\SR)$ and $\rho\in \mathcal C$, set
$$
a^{(\rho)}:=(\D_1^{(\rho)},\dots,\D_n^{(\rho)})\in M(\SC).
$$

Let  $b=(\G_1,\dots,\G_n)\in M(\SC)$ and suppose that $\G_i$ and $\G_j$ are linked for some $1\leq i<j\leq n$. Denote by $c$ the multiset that we get by replacing $\G_i$ and $\G_j$ by $\G_i\cup \G_j$ and $\G_i\cap \G_j$ in $b$ (we omit $\emptyset$ if if $\G_i\cap \G_j=\emptyset$). Then we write
$$
c\prec b.
$$
For $b_1,b_2\in M(\SC)$ we write $b_1\leq b_2$ if $b_1=b_2$, or if there exist $c_1,\dots,c_k\in M(\SC)$, with $k\geq 2$ such that
$$
b_1=c_1\prec c_2\prec\dots\prec c_k=b_2.
$$
Then $\leq$ is an ordering on $M(\SC)$.

For $\G\in\SC$ we define $\text{supp}(\G)$ to be $\G$, but considered as an element of $M(\mathcal C$). For $a=(\G_1,\dots,\G_n)\in M(\SC)$ we define
$$
\text{supp}(a)=\sum_{i=0}^n \text{supp}(\G_i)\in M(\mathcal C).
$$

The contragredient representation of $\pi$ is denoted by $\tilde \pi$. For $\D\in \SC$, set $\tilde \D:=\{\tilde \rho;\rho\in \D\}$. If $a=(\D_1,\dots,\D_k)\in M(\SC)$, then we put
$$
\tilde a=(\tilde \D_1,\dots,\tilde \D_k).
$$

\subsection{Classifications of non-unitary duals} 

Let $\D=\{\rho,\nu_\rho\rho, \dots,\nu_\rho^n\rho\}\in\SC.$
Then the representation
$$
\rho\t\nu_\rho^\rho\t \dots\t\nu_\rho^n\rho
$$
has a unique irreducible subrepresentation, which is denoted by 
$$
\z(\D),
$$
 and a unique irreducible quotient, which is denoted by 
 $$
 \d(\D).
$$

Let $a=(\D_1,\dots,\D_n)\in M(\SC)$. 
We choose an  enumeration of $\D_i$'s 
such that for all $i,j\in \{1,2,\dots,n\}$ the following  holds:

\begin{center}
if $\D_{i}\rightarrow \D_{j} $,  
then $j<i$.
\end{center}

Then the representations
$$
\aligned
&\zeta(a):= \z(\D_1)\t \z(\D_2)\t\dots\t \z(\D_n),
\\
& \lambda(a):= \d(\D_1)\t \d(\D_2)\t\dots\t \d(\D_n)
\endaligned
$$
are determined by $a$ up to an isomorphism (i.e., their isomorphism classes do not depend on the enumeration which satisfies the above condition). The representation $\zeta(a)$ has a unique irreducible subrepresentation, which is denoted by 
$$
Z(a),
$$
 while the representation $\lambda(a)$ has a unique irreducible quotient, which is denoted by 
 $$
 L(a).
 $$
  In this way we  obtain mappings
$$
Z,L:M(\SC)\rightarrow \text{Irr},
$$
which are bijections. Here, $Z$ is called Zelevinsky classification of Irr, while $L$ is called Langlands classification of Irr.
We have followed above the presentation of these classifications given by F. Rodier in \cite{Ro} in the case when $\C A$ is a field. One can find the case of non-commutative $\C A$ in \cite{T-div-a} and \cite{MS}. In \cite{T-div-a}, there is only the case of Langlands classification, while in \cite{MS} are both classifications (proofs in \cite{MS} are completely local). 

For contragredient representations, we have
$$
L(a)\tilde{\ }=L(\tilde a) \text{ \ and \ } Z(a)\tilde{\ }=Z(\tilde a).
$$

Denote by $ \DS$ the set of all essentially square integrable modulo center classes in Irr$\backslash \hat G_0$, and by $ \DS^u$ the subset of all unitarizable classes in $ \DS$ (i.e., those  having  unitary central character).
The mapping
\begin{equation}
\label{esi}
(\rho,n)\mapsto \d(\D[n]^{(\rho)}), \quad \mathcal C\t \Z_{\geq1}\ra  \DS
\end{equation}
is a bijection.

 If $\d= \d(\D[n]^{(\rho)})\in \DS$, then we denote
$$
\nu_\d=\nu_\rho
$$
(we could define $\nu_\d$ in the same way as $\nu_\rho$).

For $\d\in  \DS$ define $\d^u\in \DS^u$ and $e(\d)\in\R$ by the following requirement:
$$
\d=\nu^{e(\d)}_\d\d^u.
$$
Let $\d\in M( \DS)$. We can choose an enumeration of elements of $d$ which satisfies:
$$
e(\d_1)\geq e(\d_2) \geq \dots \geq e(\d_n).
$$
Let
$$
\lambda(d)=\d_1\t\d_2\t\dots\t\d_n.
$$
Then the representation $\lambda(d)$ has a unique irreducible quotient, denoted by $L(d)$. Again $L: M( \DS) \ra \text{Irr}$ is a bijection, and it is one of the possible ways to express the Langlands classification in this case.

The representations
$$
u(\d,n)=L((\nu_\d^{\frac{n-1}2}\d, \nu_\d^{\frac{n-1}2-1}\d,\dots,\nu_\d^{-\frac{n-1}2}\d)),\ \ \ \d\in  \DS,
$$
 are essentially unitarizable (i.e., they become unitarizable after a twist by the appropriate character; see \cite{BR-Tad} and \cite{BHLS}).

\subsection{Duality - Zelevinsky involution} Define a mapping
$$
^t:\text{Irr} \ra \text{Irr}
$$
by $Z(a)^t=L(a), a \in M(\SC)$. Extend $^t$ additively to $R$. Clearly, $^t$ is a positive mapping, i.e., satisfies: $r_1\leq r_2\implies r_1^t\leq r_2^t$. A non-trivial fact is that $^t$ is also multiplicative, i.e., a ring homomorphism (see \cite{Au}, \cite{Ro} and \cite{ScSt}). Further, $^t$ is an involution. For $ a \in M(\SC)$ we define $a^t\in M(\SC)$ by the requirement
$$
(L(a))^t=L(a^t).
$$
We could also use the Zelevinsky classification to define $^t:M(\SC)\ra M(\SC)$, and we would get the same involutive mapping.

One can find more information about the involution in \cite{Ro}.

\subsection{Algorithm of C. M\oe glin and J.-L. Waldspurger} 
Let $a\in M(\SR)$ and $\rho\in\SC$. Then there exists $a^t\in M(\SR)$, independent of $\rho$, such that
$$
(a^{(\rho)})^t= (a^t)^{(\rho)}.
$$
Now we recall  the combinatorial algorithm from \cite{MoeW-alg} describing $a^t$.

Consider segments $\D$  in $a$ with maximal $e(\D)$. Among these segments, choose one with maximal $b(\D)$. Denote it by $\D_1$, and denote its end by $x$. This will be called the first stage of the algorithm.

For the following stage,  consider segments $\D$ in $a$  which end at $x-1$, and which are linked with $\D_1$ (i.e. which precede $\D_1$). Among them, if  such segments exist, choose one with maximal $b(\D)$. Denote it by $\D_2$.

One continues this procedure with ends $x-2$, $x-3$, etc., as long as it is possible. The segments considered in this procedure are $\D_1, \dots, \D_k$ ($k\geq 1$).
Let 
$$
\G_1=[x-k+1,x]\in  M(\Z).
$$
This set of stages of the algorithm will be called the first step of the algorithm.

 Let $a^\la$ be the multiset of  $M(\SR)$ which we get from $a$ by replacing each $\D_i$ by $\D_i^-$, $i=1,\dots ,k$ (we simply omit those $\D_i^-$ for which  $\D_i^-=\emptyset$).
 If $a^\la$  is non-empty, we now repeat  the above procedure with $a^\la$. In this way we get a segment $\G_2$ and $(a^\la)^\la\in M(\SR)$.

Continuing this procedure as long as possible, we get $\G_1,\dots,\G_m\in \mathcal S(\R)$. Then by \cite{MoeW-alg} (see also \cite{BR-inv}) we have 
$$
a^t= (\G_1,\dots,\G_m).
$$
This algorithm will be denoted by
$$
\text{MWA}^{\la}.
$$
\begin{definition}
The set of stages of the algorithm, which end with some  segment $\G_i$, will be called a step of the algorithm.
\end{definition}

\begin{remark} We shall often use the following simple facts in our applications of the algorithm:
\label{rem-alg}
\begin{enumerate}
\item If at same stage of the algorithm we have used a segment $[X,Y]$, and if at this stage we have at disposal  at least one copy of the segment $[X,Y]_{-1}$, then in the  following stage we must use  one copy of the segment $[X,Y]_{-1}$, and both stages are part of the same step.

\item If at same stage of the algorithm we have used segment $[X,Y]$ with the property that $X$ is $\leq$  the beginnings of all the segments that we have at disposal at this stage, then this is the last stage of the step that we perform.

\item If at same stage of the algorithm we have used a segment $[X,Y]$, then each segment that we will  use in the remaining stages of this step must have cardinality 
at least $Y-X+1$.

\end{enumerate}
\end{remark}

\subsection{Dual algorithm}
Extend the mapping given by $\pi\mapsto \tilde \pi$ on the equivalence classes of irreducible representations to an a homomorphism of $R$ as an additive group (it is also a ring homomorphism). We denote this map (again) by  $\tilde{\ }:R \ra R$. Then it is a ring homomorphism. One directly sees  on generators $\d(\D)$, $\D\in \SC$, that this homomorphism commutes with the Zelevinsky involution. Thus
$$
(L(a)\tilde{\ })^t =(L(a)^t)\tilde{\ }, \ a\in M(\SC).
$$
Therefore, for $a\in M(\SR) $ and $\rho \in \mathcal C$, we can apply the above algorithm to $L(a^{(\rho)})\tilde{\ }$, and after that apply once again the contragredient mapping. Since the contragredient mapping is an involution, we get the Zelevinsky involution. This gives the following (dual) version of the above algorithm. 

Consider segments $\D$ in $a$ with minimal $b(\D)$. Among them,  pick a segment with minimal $e(\D)$. Denote it by $\D_1$, and its beginning with $x$. Now consider segments $\D$  in $a$ which begin by $x+1$ and which are linked with $\D_1$, if any such segment exists. Among these segments choose one with minimal $e(\D)$. Denote it by $\D_2$.
One continues this procedure with beginnings $x+2$, $x+3$, etc., as long as it is possible. The segments that have shown up  in this procedure are denoted by $\D_1, \dots, \D_k$ ($k\geq 1$).
Put 
$$
\G_1=[x,x+k-1]\in  M(\Z).
$$
Let $^\ra a$ be the multiset of  $M(\SR)$ which we get from $a$ by replacing each $\D_i$ by $^-\D_i$ (if $^-\D_i=\emptyset$, we simply omit it).
 If $^\ra a$  is non-empty, we repeat  the above procedure with $^\ra a$. In this way we get $\G_2$ and $^\ra(^\ra a)$.
Continuing this procedure as long as possible, we get $\G_1,\dots,\G_m\in \mathcal S(\R)$. Then 
$$
a^t= (\G_1,\dots,\G_m).
$$
This algorithm will be denoted by
$$
^{\ra}\text{MWA}.
$$

We shall usually apply the above algorithm(s) to elements of $M(\SC)$ in an obvious way. It is easy to show  that
$$
(a(n,d)^{(\rho)})^t=a(d,n)^{(\rho)}.
$$

\subsection{Upper bound for  the lengths of the segments in the dual  multisegment}
\label{est}
 We will later use the following observation of  C. M\oe glin and J.-L. Waldspurger from \cite{MoeW-alg} (this is  remark (P) before Theorem 1
 in \cite{Ba-Sp}): if there exists a segment $\D$ of length $m$ such that all the ends of segments in $a\in M(\SC)$ are contained in $\D$, then the length of segments in $a^t$ can be at most $m$.

Dually, we get the following observation:  If there is a segment $\D$ of length $m$ such that all the beginnings of segments in $a\in M(\SC)$ are contained in $\D$, then the length of segments in $a^t$ is at most $m$.

\subsection{Support} Let $\pi\in \text{Irr}$. Take $a\in M(\SC) $ such that $\pi=L(a)$. Then the support of $\pi$ is defined by
$$
\text{supp}(\pi)=\text{supp}(a).
$$
If we take 
 $b\in M(\SC) $ such that $\pi=Z(a)$, then $
\text{supp}(\pi)=\text{supp}(b).
$

Suppose that for a finite length representation $\pi'$, for each two irreducible subquotients $\pi_1$ and $\pi_2$ one has       $\text{supp}(\pi_1)=\text{supp}(\pi_2).$ Then we define $\text{supp}(\pi')$ to be $\text{supp}(\pi)$, where $\pi$ is (any) irreducible subquotient of $\pi'$.

\subsection{Classification of the unitary dual} 

Denote by 
$$
B_{\text{rigid}}=\{u(\d,n); \d\in  \DS^u,n\in\Z_{\geq1}\}.
$$
 and
 $$
 B=B_{\text{rigid}}\cup\{\nu_\d^\a u(\d,n)\t\nu_\d^{-\alpha}u(\d,n); u(\d,n)\in B_{\text{rigid}}, 0<\a<1/2\}.
$$
 Then the unitary dual is described  by the following\footnote{Exactelly the same theorem classifies  the unitary dual also in the archimedean case, i.e. for general linear groups over $\mathbb R, \mathbb C$ and $\mathbb H$ (see \cite{T-R-C-new}, or earlier version \cite{T-R-C-old}, and \cite{BR-arch}).}:

\begin{theorem} ( \cite{BHLS},\cite{BR-Tad},\cite{Se},\cite{T-div-a})
\begin{enumerate}
\item Let 
$\tau_1, \ldots , \tau_n \in B$.
Then the representation 
$$
\pi:=\tau_1 \times \ldots \times \tau_n
$$
 is irreducible and unitary.

\item Suppose that a representation
$\pi'$
is obtained from
$\tau_1', \ldots , \tau_{n'}' \in B$
in the same manner as
$\pi$
was obtained from
$\tau_1, \ldots , \tau_n$
in (1). Then
$\pi \cong \pi'$
if and only if $n=n'$ and the sequences
$(\tau_1, \ldots , \tau_n)$
and
$(\tau_1', \ldots , \tau_n')$
coincide after a renumeration.

\item Each irreducible  unitary representation of $GL(m,F)$, for
any $m$, can be obtained as in (1). 
\end{enumerate}
\end{theorem}

Recall that
$$
u(\D[d]^{(\rho)},n)= L(a(d,n)^{(\rho)}),
$$
and
\begin{equation}
\label{t}
Z(a(n,d)^{(\rho)})\cong L(a(d,n)^{(\rho)})
\end{equation}
 for $n,d\in\Z_{\geq1}$ and $\rho\in\mathcal C$.

\section{Some criteria for reducibility and irreducibility}

\subsection{A reducibility criterion} Let $a,b\in M( \DS)$. We know 
$$
L(a+b)\leq L(a)\t L(b).
$$
Thus
$$
L((a+b)^t)=L(a+b)^t\leq L(a)^t\t L(b)^t.
$$
 Also 
 $$
 L(a^t+b^t)\leq L(a^t)\t L(b^t)=L(a)^t\t L(b)^t.
 $$
This implies the following well known reducibility criterion
\begin{equation}
\label{RC}
 \text{If $(a+b)^t\ne a^t+b^t$, then $ L(a)\t L(b)$ reduces}.
\end{equation}

In general,   $(a+b)^t= a^t+b^t$ does not imply the irreducibility of $ L(a)\t L(b)$.

\subsection{Irreducibility criterion of I. Badulescu} Suppose that $a_1,a_2\in M(\SC)$ satisfy $(a_1+a_2)^t=a_1^t+a_2^t$ (if this is not the case, then  $L(a_1)\t L(a_2)$ is reducible). Assume that for each $b\in M(\SC)$ the following implication  holds 
\begin{equation}
\label{BRC}
b< a_1+a_2 \implies b^t\not < (a_1+a_2)^t.
\end{equation}
Then $L(a_1)\t L(a_2)$ is irreducible\footnote{This way of proving irreducibility was used in the proof of Theorem 1 in \cite{Ba-Sp} (as far as we know, this is the first case where this simple idea to prove irreducibility was used).}.

In general, irreducibility of $L(a_1)\t L(a_2)$ does not imply that \eqref{BRC} holds.

For the convenience of the reader, we  repeat the argument from \cite{Ba-Sp}.

Suppose that $L(a_1)\t L(a_2)$ is reducible. Then  in $R$ we have
$$
L(a_1)\t L(a_2)=L(a_1+a_2)+\sum_{i=1}^k n_i L(b_i),
$$
where $n_i\in \Z_{\geq1}$, $k\geq 1$ and 
$$
b_i< a_1+a_2, \quad i=1,\dots,k.
$$

Applying the Zelevinsky involution, we get (in $R$)
$$
L(a_1^t)\t L(a_2^t)=L(a_1^t+a_2^t)+\sum_{i=1}^k n_i 
L(b_i)^t
$$
(here we have used $(a_1+a_2)^t=a_1^t+a^t_2$). The properties of the Langlands classification imply
$$
b_1^t,\dots,b_k^t <a_1^t+a_2^t.
$$
This contradicts the assumptions of the criterion.

\subsection{Contacting and crossing} To present the next criterion, we need the following:
\begin{definition}
Let $a,b\in M( \SC)$. We say that $a$ is in contact with $b$ (or simply that $a$ contacts $b$), if there exist  segments $\D_1$ and $\D_2$ in $a$ and $b$ respectively, which are juxtaposed (two non-empty segments are called juxtaposed if they are disjoint and if their union is a segment).

We say that $a$ and $b$ are crossed if $a$ contacts $b^t$, and $a^t$ contacts $b$.

We  say that irreducible representations are in contact (resp., are crossed) if the multisegments corresponding to them with respect to the Langlands classification are in contact (resp., are crossed). 
\end{definition}

\begin{remark}
Observe that $a$ is in contact with $b$ if and only if $\tilde a$ is in contacts with $\tilde b$. Further, the action of Zelevinsky involution and the contragredient mapping commute on segments, i.e., $(\tilde \D)^t=(\D^t)\tilde{\ }$, which implies that they commute on $R$. From this, it  follows easily that
$$
(\tilde a)^t=(a^t)\tilde{\ }.
$$
Therefore, 
$$
\text{$a$ and $b$ are crossed if and only if $\tilde a$ and $\tilde b$ are crossed.}
$$

Clearly
$
L(a)\t L(b)$ reduces  $\iff  L(\tilde a)\t L(\tilde b)$ reduces $\iff L(a^t)\t L(b^t) $ reduces $ \iff L(\tilde a^t)\t L(\tilde b^t) $ reduces.

These equivalences  also hold for the Zelevinsky classification (and $L(a)\t L(b)$  reduces $\iff  Z(a)\t Z(b)$ reduces).
\end{remark}

\subsection{Irreducibility criterion of I. Badulescu, E. Lapid and A. M\'inguez}
This criterion (which is Theorem 3.9 of \cite{BLM})
for $a,b \in M( \DS)$ says the following:
\begin{equation}
\label{BLM}
 \text{If $a$ is not in contact with $b$, then $ L(a)\t L(b)^t$ is irreducible}.
\end{equation}
We get directly  from this the following criterion:
\begin{equation}
\label{BLM'} 
 \text{If  $a$ and $b^t$, or $a^t$ and $b$, are not in contact, then $ L(a)\t L(b)$ is irreducible.}
\end{equation}
In other words:
\begin{equation}
\label{BLM''}
 \text{If $ L(a)\t L(b)$ is reducible, then  $a$ and $b$ are crossed}.
\end{equation}

\section{Contacts of non-induced  essentially unitarizable representations}

We  now describe when two  representations $L(a(n,d)^{(\nu^k_\rho \rho)}$ and $L(a(m,d)^{(\nu^l_\rho\rho)})$, supported on the same cuspidal line, are in contact. 

\subsection{Contacts among $L(a(n,d)^{(\rho)})$'s}
Let
$$
\pi_1=L(a(n,d)^{(\nu_{\rho'}^k\rho')}) , \qquad \pi_2= L(a(m,e)^{(\nu_{\rho'}^l\rho')}), \quad k,l\in \Z
$$
After twisting $\rho'$, we can write
$$
L(a(n,d)^{(\nu_{\rho'}^k\rho')})= L([1,d]^{(\rho)},[2,d+1]^{(\rho)}],\dots,[n,d+n-1]^{(\rho)}), \hskip10mm
$$
$$
L(a(m,e)^{(\nu_{\rho'}^l\rho')})=  \hskip90mm
$$
$$
\hskip30mm L([b,b+e-1]^{(\rho)},[b+1,b+e]^{(\rho)},\dots,[b+m-1,b+m-1+e-1]^{(\rho)}),
$$
for some $\rho\in \C C$. 

In what follows we  assume 
$$
b\in\Z
$$
(otherwise, $\pi_1\t\pi_2$ is always irreducible).

Now $\pi_1$ and $\pi_2$ are in  contact (see 3.3) if and only if
\begin{equation}
\label{condition}
[1,n]\cap [b+e,b+m-1+e]\ne\emptyset \text{ \ \ or \ \ } [d+1,d+n]\cap [b, b+m-1]\ne \emptyset.
\end{equation}

We can graphically interpret this by the following drawing:

\hskip25mm
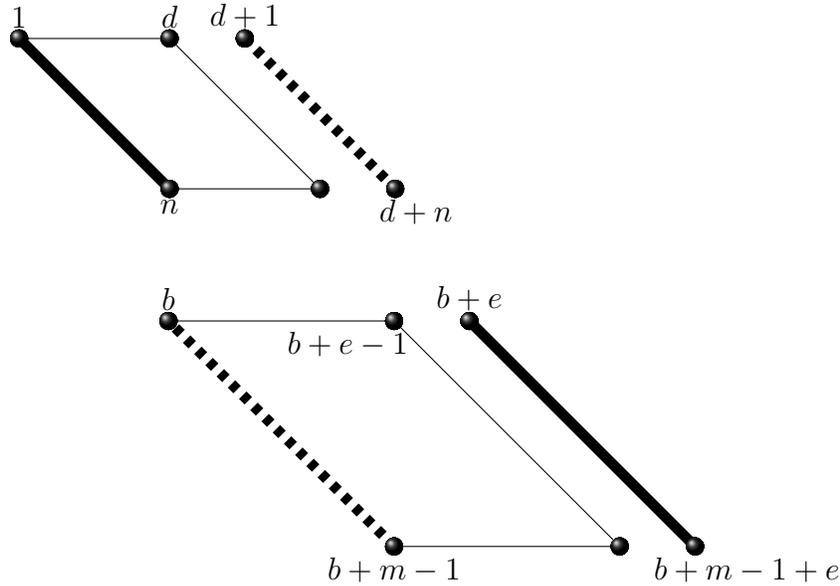
\begin{tikzpicture} 
\draw[line width=4pt] (3,1) -- (1,3);
\path[draw]  (1,3) -- (3,3) -- (5,1) -- (3,1);
\draw[style=dashed,line width=4pt] (4,3) -- (6,1);
\draw (3,1) node[below] {$n$};
\draw (1,3) node[above] {$1$};
\draw (3,3) node[above] {$d$};
%\draw (5,1) node[below] {$d+n-1$};
\draw (4,3) node[above] {$  d+1$};
\draw (6,1) node[below] {$\ \ \ \ d+n$};
\shade[shading=ball, ball color=black] (3,1) circle (.12);
\shade[shading=ball, ball color=black] (1,3) circle (.12);
\shade[shading=ball, ball color=black] (3,3) circle (.12);
\shade[shading=ball, ball color=black] (5,1) circle (.12);
\shade[shading=ball, ball color=black] (4,3) circle (.12);
%\shade[shading=ball, ball color=black] (8,1) circle (.12);
\shade[shading=ball, ball color=black] (6,1) circle (.12);
 \end{tikzpicture}

 \hskip45mm
\begin{tikzpicture} 
\draw[style=dashed,line width=4pt] (4,1) -- (1,4);
\path[draw]  (1,4) -- (4,4) -- (7,1) -- (4,1);
\draw[line width=4pt] (5,4) -- (8,1);
\draw (4,1) node[below] {$b+m-1$};
\draw (1,4) node[above] {$b$};
\draw (4,4) node[below] {$b+e-1 \ \ \ \  \ \ \ \ \ $};
%\draw (5,1) node[below] {$d+n-1$};
\draw (5,4) node[above] {$b+e$};
\draw (8,1) node[below] {$\  \ \ \ \ \ \ \ \  \ b+m-1+e$};
\shade[shading=ball, ball color=black] (4,1) circle (.12);
\shade[shading=ball, ball color=black] (1,4) circle (.12);
\shade[shading=ball, ball color=black] (4,4) circle (.12);
%\shade[shading=ball, ball color=black] (4,1) circle (.12);
\shade[shading=ball, ball color=black] (5,4) circle (.12);
\shade[shading=ball, ball color=black] (8,1) circle (.12);
\shade[shading=ball, ball color=black] (7,1) circle (.12);
 \end{tikzpicture}

Looking at the above drawing, we have contact between $\pi_1$ and $\pi_2$
  if and only if the intersection of the (projections to the horizontal axis  of) bold lines is non-empty, or the  intersection of (projections of) dashed lines is non-empty.
 
\subsection{Some remarks regarding irreducibility} We  study when $\pi_1\t\pi_2$ reduces. Since $\pi_1\t\pi_2$ reduces if and only if $\pi_2\t\pi_1$ reduces, without lost of generality we can always enumerate $\pi_i$'s in such a way that 
$$
1\leq b.
$$ 
Since we are  interested in reducibility of $\pi_1\t \pi_2$, we  consider only the case when the union of underlying sets of supports of $\pi_1$ and $\pi_2$ is a segment (if it is not, then $\pi_1\t \pi_2$ is irreducible). Therefore, we  assume
$$
b\leq d+n.
$$
We retain the assumptions 
\begin{equation}
\label{ass}
1\leq b\leq d+n
\end{equation}
in what follows.

Observe that the first condition from \eqref{condition} for $\pi_i$'s to be in contact, $[1,n]\cap [b+e,b+m-1+e]\ne\emptyset$,  is now equivalent to
$
b+e\leq n$, $\text{   i.e., \ \ } $
$$
b+e-1< n .
$$

\subsection{Another notation}
We  denote

$
\qquad \quad A_1=1,\qquad \quad\quad\qquad\qquad B_1=d,
$

$
\qquad \qquad  C_1=n,\qquad \quad\quad\qquad\qquad D_1=n+d-1,
$

$
 \qquad \quad \qquad A_2=b,\quad\qquad \qquad  \qquad \quad B_2=b+e-1,
$

$
\qquad  \quad \quad \qquad C_2=b+m-1,\qquad\qquad  D_2=b+m-1+e-1.
$

Obviously, $A_1\leq  B_1,C_1\leq D_1$, $A_2\leq  B_2,C_2\leq D_2$, and
$$
B_1-A_1=D_1-C_1,
$$
$$
B_2-A_2=D_2-C_2.
$$
The previous  assumption, $1\leq b\leq d+n$, now   becomes
\begin{equation}
\label{ASS}
A_1
\leq 
A_2\leq D_1+1.
\end{equation}

Now we have contact if and only if
$$
[A_1,C_1]\cap [B_2+1,D_2+1]\ne\emptyset \text{ \ \ or \ \ } [B_1+1,D_1+1]\cap [A_2, C_2]\ne \emptyset.
$$

The previous drawing now corresponds to the following drawing in the new notation:

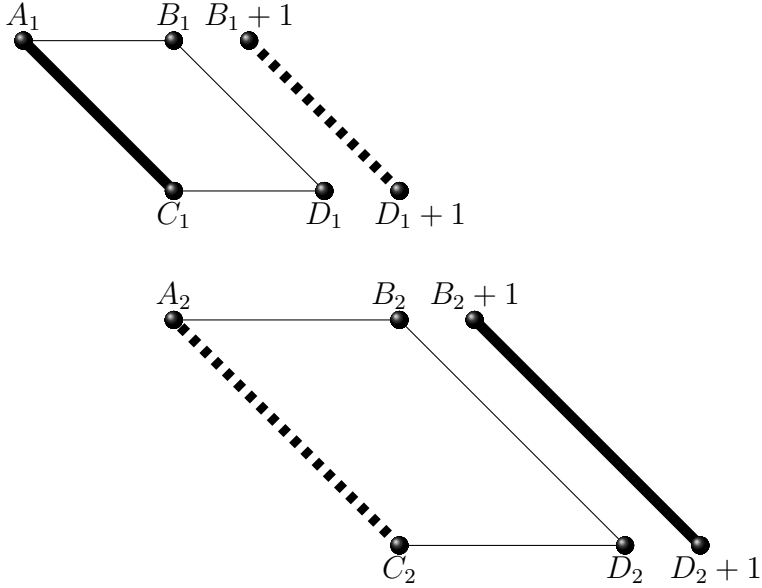
\begin{tikzpicture} 
\draw[line width=4pt] (3,1) -- (1,3);
\path[draw]  (1,3) -- (3,3) -- (5,1) -- (3,1);
\draw[style=dashed,line width=4pt] (4,3) -- (6,1);
\draw (3,1) node[below] {$C_1$};
\draw (1,3) node[above] {$A_1$};
\draw (3,3) node[above] {$B_1$};
\draw (5,1) node[below] {$D_1$};
\draw (4,3) node[above] {$  B_1+1$};
\draw (6,1) node[below] {$\ \ \ \ D_1+1$};
\shade[shading=ball, ball color=black] (3,1) circle (.12);
\shade[shading=ball, ball color=black] (1,3) circle (.12);
\shade[shading=ball, ball color=black] (3,3) circle (.12);
\shade[shading=ball, ball color=black] (5,1) circle (.12);
\shade[shading=ball, ball color=black] (4,3) circle (.12);
\shade[shading=ball, ball color=black] (6,1) circle (.12);
 \end{tikzpicture}

 \hskip20mm
\begin{tikzpicture} 
\draw[style=dashed,line width=4pt] (4,1) -- (1,4);
\path[draw]  (1,4) -- (4,4) -- (7,1) -- (4,1);
\draw[line width=4pt] (5,4) -- (8,1);
\draw (4,1) node[below] {$C_2$};
\draw (1,4) node[above] {$A_2$};
\draw (4,4) node[above] {$B_2\ \ $};
\draw (7,1) node[below] {$D_2$};
\draw (5,4) node[above] {$B_2+1$};
\draw (8,1) node[below] {$\ \ \ D_2+1$};
\shade[shading=ball, ball color=black] (4,1) circle (.12);
\shade[shading=ball, ball color=black] (1,4) circle (.12);
\shade[shading=ball, ball color=black] (4,4) circle (.12);
\shade[shading=ball, ball color=black] (5,4) circle (.12);
\shade[shading=ball, ball color=black] (8,1) circle (.12);
\shade[shading=ball, ball color=black] (7,1) circle (.12);
 \end{tikzpicture}

We have contact if and only if the (projections to the horizontal axis of) bold lines have non-empty intersection, or the dashed lines have non-empty intersection.

Since we assume $A_1\leq A_2$, the  first condition  is equivalent to
$
B_2+1\leq C_1,$
i.e.,
$
B_2< C_1 .
$
Therefore, we have contact if and only if
\begin{equation}
\label{CON}
B_2<C_1 \text{ \ \ or \ \ } [B_1+1,D_1+1]\cap [A_2, C_2]\ne \emptyset.
\end{equation}

We end this section with the following

\begin{lemma}
\label{one-condition} Suppose that the underlying set of $\text{supp}(L(a(n,d)^{(\nu_{\rho'}^k\rho')})$ precedes the underlying set of  $\text{supp}(L(a(m,e)^{(\nu_{\rho'}^l\rho')})$.  Then $\pi_1$ is in contact with $\pi_2$ if and only if the dashed segments intersect. 
\end{lemma}

\begin{proof}
The assumption of the lemma that underlying sets of $\text{supp}(L(a(n,d)^{(\nu_{\rho'}^k\rho')})$ and $\text{supp}(L(a(m,e)^{(\nu_{\rho'}^l\rho')})$ are linked,  implies
$$
A_1<A_2, D_1<D_2, A_2\leq D_1+1.
$$
Suppose that the $\pi_i$'s are in contact, and that the intersection of bold segments is non-empty. This implies
$$
B_2+1\leq C_1 \text{ i.e. } B_2<C_1.
$$
We shall now show   that
\begin{equation}
\label{seg}
 [B_1+1,D_1+1]\cap [A_2, C_2]
 \end{equation}
  is also non-empty  (which is the intersection of the dashed segments).
We consider  two cases.
\begin{enumerate}
\item Suppose $B_1+1\leq A_2$. Then \eqref{seg} is non-empty if and only if $[A_2,D_1+1]\ne \emptyset$. Clearly, this is the case if and only if $A_2\leq D_1+1$. We know that this holds. Therefore, \eqref{seg} is non-empty.

\item Now, suppose  $A_2<B_1+1$. Then \eqref{seg} is non-empty if and only if $[B_1+1,C_2]\ne \emptyset$, which  is the case if and only if $B_1+1 \leq C_2$. We show below that this holds.

Observe that $A_1<A_2$, $D_1<D_2$ and $B_2+1\leq C_1$ imply $A_1+B_2+D_1+1< A_2+C_1+D_2$, i.e.,  $A_1-C_1+D_1+1< A_2-B_2+D_2$. From $A_i+D_i=B_i+C_i$, $i=1,2$, we get $B_1+1<C_2$. This obviously  implies the inequality $\leq$, which we wanted to prove. Therefore, \eqref{seg} is not empty.

\end{enumerate}
The proof  of the lemma is now complete.
\end{proof}

\section{Reducibility in the case of linked underlining sets of supports}\label{linked}

We continue with the notation introduced in the previous section.

\subsection{Reducibility criterion in the linked case}

\begin{proposition}
Suppose that the underlying sets of  
$$
\text{supp}(L(a(n,d)^{(\nu_{\rho'}^k\rho')})\text{ \ and \ }\text{supp}(L(a(m,e)^{(\nu_{\rho'}^l\rho')})
$$
 are linked segments. Then 
\begin{equation}
\label{prod}
L(a(n,d)^{(\nu_{\rho'}^k\rho')}) \t L(a(m,e)^{(\nu_{\rho'}^l\rho')})
\end{equation}
 reduces if and only if $L(a(n,d)^{(\nu_{\rho'}^k\rho')})$ and $ L(a(m,e)^{(\nu_{\rho'}^l\rho')})$ are crossed.
\end{proposition}

\begin{proof} Thanks to \eqref{BLM''}, we know that the reducibility of \eqref{prod} implies that the corresponding multisegments are crossed. We need to prove the opposite implication, i.e. that if we have crossed multisegments in the lemma, then we have reducibility.

 Let
$$
\pi_1=L(a(n,d)^{(\nu_{\rho'}^k\rho')}) , \qquad \pi_2= L(a(m,e)^{(\nu_{\rho'}^l\rho')}),
$$
$$
a_1=a(n,d)^{(\nu^k_{\rho'}\rho')}, \qquad a_2=a(m,e)^{(\nu_{\rho'}^l\rho')}.
$$
Here we can write 
$$
L(a(n,d)^{(\nu_{\rho'}^k\rho')}= L([1,d]^{(\rho)},[2,d+1]^{(\rho)}],\dots,[n,d+n-1]^{(\rho)}]),
$$
$$
L(a(m,e)^{(\nu_{\rho'}^l\rho')})=  L([b,b+e-1]^{(\rho)},[b+1,b+e]^{(\rho)},\dots,[b+m-1,b+m-1+e-1]^{(\rho)}]),
$$
for some $\rho\in \C C$, where 
$
1<b,
$
$
d+n<b+m-1+e,
$
and
$
b\leq d+n.
$

We introduce $A_i,B_i,C_i,D_i$ by the same formulas as in 4.2.
The linking condition in this notation  is 
$$
A_1<A_2,
$$ 
$$
D_1<D_2,
$$
$$
A_2\leq D_1+1.
$$
Further, by Remark \ref{one-condition}, the crossing condition is equivalent to
$$
[B_1+1,D_1+1]\cap [A_2,B_2]\ne \emptyset
$$
and
$$
[C_1+1,D_1+1]\cap [A_2,C_2]\ne \emptyset.
$$
Since $A_2\leq D_1+1$, the above two conditions are equivalent to
$$
B_1+1\leq B_2
$$
and
$$
C_1+1\leq C_2.
$$

We consider several cases.

\begin{enumerate} 

\item Let $D_1<B_2$.

\begin{enumerate}

\item Suppose $C_1<A_2.$

We first  illustrate   the situation graphically:
\begin{equation} 
\label{g1}
\xymatrix@C=.6pc@R=.1pc
{ 
a_1:& \bullet & \bullet \ar @{-}[l]  & \bullet \ar @{-}[l]  & \bullet \ar @{-}[l]  & \bullet \ar @{-}[l] 
\\ 
&& \bullet & \bullet \ar @{-}[l]  & \bullet \ar @{-}[l]  & \bullet \ar @{-}[l]  & \bullet \ar @{-}[l] 
\\
a_2: & &&&  & \bullet & \bullet \ar @{-}[l]  & \bullet \ar @{-}[l]  & \bullet \ar @{-}[l]  & \bullet \ar @{-}[l]  
  \\
  &&&& && \bullet & \bullet \ar @{-}[l]  & \bullet \ar @{-}[l]  & \bullet \ar @{-}[l]  & \bullet \ar @{-}[l]  
  \\
 & & &&& && \bullet & \bullet \ar @{-}[l]  & \bullet \ar @{-}[l]  & \bullet \ar @{-}[l]  & \bullet \ar @{-}[l]  
  \\
 && & &&& && \bullet & \bullet \ar @{-}[l]  & \bullet \ar @{-}[l]  & \bullet \ar @{-}[l]  & \bullet \ar @{-}[l]  
} 
\end{equation}
\begin{equation} 
\label{g2}
\xymatrix@C=.6pc@R=.1pc
{ 
a_1^t: & \bullet & \bullet & \bullet & \bullet & \bullet
\\ 
&& \bullet \ar[lu] & \bullet\ar[lu] & \bullet\ar[lu] & \bullet\ar[lu] & \bullet\ar[lu] 
\\
a_2^t: & &&&  & \bullet & \bullet & \bullet &\bullet &\bullet
  \\
 & & &&& & \bullet \ar[lu]& \bullet\ar[lu] & \bullet\ar[lu] & \bullet\ar[lu] & \bullet\ar[lu] 
  \\
& & &&& & & \bullet \ar[lu]& \bullet\ar[lu] & \bullet\ar[lu] & \bullet\ar[lu] & \bullet\ar[lu] 
  \\
& && &&& & & \bullet \ar[lu]& \bullet\ar[lu] & \bullet\ar[lu] & \bullet\ar[lu] & \bullet\ar[lu] 
} 
\end{equation}
\begin{equation} 
\label{g3}
\xymatrix@C=.6pc@R=.1pc
{ 
(a_1+a_2)^t & \bullet & \bullet & \bullet & \bullet & \bullet
\\ 
&& \bullet \ar[lu] & \bullet\ar[lu] & \bullet\ar[lu] & \bullet\ar[lu] & \bullet\ar[lu] 
\\
 & &&&  & \bullet \ar[lu] & \bullet \ar[lu]  & \bullet \ar[lu]  &\bullet &\bullet
  \\
 & & &&& & \bullet \ar[lu]& \bullet\ar[lu] & \bullet\ar[lu] & \bullet\ar[lu] & \bullet\ar[lu] 
  \\
& & &&& & & \bullet \ar[lu]& \bullet\ar[lu] & \bullet\ar[lu] & \bullet\ar[lu] & \bullet\ar[lu] 
  \\
& && &&& & & \bullet \ar[lu]& \bullet\ar[lu] & \bullet\ar[lu] & \bullet\ar[lu] & \bullet\ar[lu] 
} 
\end{equation}
We now show $a_1^t+a_2^t\ne (a_1+a_2)^t$. Then the reducibility criterion \eqref{RC} implies  reducibility. Observe that $a_1^t+a_2^t=a(d,n)^{(\nu^k_{\rho'}\rho')}+a(e,m)^{(\nu_{\rho'}^l\rho')}$. Therefore in this multisegment, there are only segments of length $n$ and $m$. The assumptions  $C_1<A_2$ and $D_1< B_2$, together with MWA$^\la$ directly imply that $(a_1+a_2)^t$ will have at least one segment of length $n+m$ (see the graphical interpretation). Namely, in the first $B_2-D_1-1$ steps of the algorithm, we get segments of length $m$, and in the following step we get a segment of length $m+n$. This completes the proof of reducibility in this case.

\item Now suppose  $A_2\leq C_1$. 

We again  illustrate   the situation graphically:
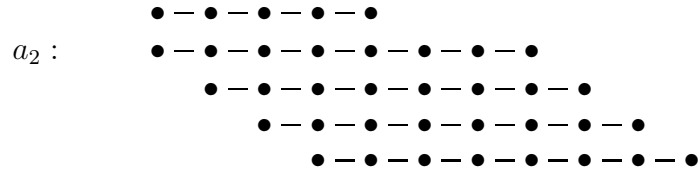
\begin{equation} 
\label{g4}
\xymatrix@C=.6pc@R=.1pc
{ 
a_1:& \bullet & \bullet \ar @{-}[l]  & \bullet \ar @{-}[l]  & \bullet \ar @{-}[l]  & \bullet \ar @{-}[l]  
\\ 
&&    \bullet & \bullet \ar @{-}[l]  & \bullet \ar @{-}[l]  & \bullet \ar @{-}[l]  & \bullet \ar @{-}[l]  
\\
a_2: &   &  \bullet & \bullet \ar @{-}[l]  & \bullet \ar @{-}[l]  &  \bullet  \ar @{-}[l]  & \bullet \ar @{-}[l]  & \bullet \ar @{-}[l]  & \bullet \ar @{-}[l]  & \bullet \ar @{-}[l]  
  \\
  &&&  \bullet & \bullet \ar @{-}[l]  & \bullet \ar @{-}[l]  &  \bullet  \ar @{-}[l]   & \bullet \ar @{-}[l]  & \bullet \ar @{-}[l]  & \bullet \ar @{-}[l]  & \bullet \ar @{-}[l]  
  \\
 & & &&  \bullet & \bullet \ar @{-}[l]  & \bullet \ar @{-}[l]  &  \bullet  \ar @{-}[l]  & \bullet \ar @{-}[l]  & \bullet \ar @{-}[l]  & \bullet \ar @{-}[l]  & \bullet \ar @{-}[l]  
  \\
 && & &&  \bullet & \bullet \ar @{-}[l]  & \bullet \ar @{-}[l]  &  \bullet \ar @{-}[l]   & \bullet \ar @{-}[l]  & \bullet \ar @{-}[l]  & \bullet \ar @{-}[l]  & \bullet \ar @{-}[l]  
} 
\end{equation}
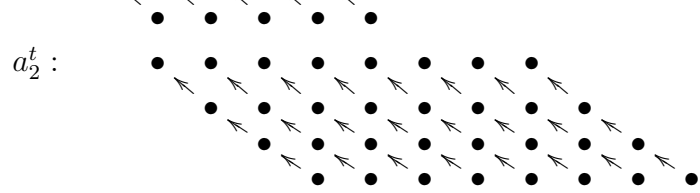
\begin{equation} 
\label{g5}
\xymatrix@C=.6pc@R=.1pc
{ 
a_1^t: & \bullet & \bullet & \bullet & \bullet & \bullet 
\\ 
&& \bullet \ar[lu] & \bullet\ar[lu] & \bullet\ar[lu] & \bullet\ar[lu] & \bullet\ar[lu]       
\\
a_2^t: &   & \bullet & \bullet & \bullet &\bullet &\bullet & \bullet &\bullet &\bullet
  \\
 & &  & \bullet \ar[lu]& \bullet\ar[lu] & \bullet\ar[lu] & \bullet\ar[lu] & \bullet\ar[lu]  & \bullet\ar[lu] & \bullet\ar[lu] & \bullet\ar[lu] 
  \\
& &  & & \bullet \ar[lu]& \bullet\ar[lu] & \bullet\ar[lu] & \bullet\ar[lu] & \bullet\ar[lu] & \bullet\ar[lu] & \bullet\ar[lu] & \bullet\ar[lu] 
  \\
& &&  & & \bullet \ar[lu]& \bullet\ar[lu] & \bullet\ar[lu] & \bullet\ar[lu] & \bullet\ar[lu] & \bullet\ar[lu] & \bullet\ar[lu] & \bullet\ar[lu] 
} 
\end{equation}
Observe that in $a_1^t+a_2^t$ there is only one segment starting at exponent $A_1=1$. This is the segment $[A_1,C_1]^{(\rho)}=[1,n]^{(\rho)}$.

Now applying  $^\ra$MWA, starting with exponent $1$, we get exponents $2$, $3$, \dots, $n$. Since $n=C_1<C_2$, we can find a segment in $a_2$ starting with $n+1$. The assumption  $D_1<B_2$ implies that this segment is linked with the previous segment used in the algorithm. This implies that the (unique) segment in $(a_1+a_2)^t$ starting with $A_1$ is not $[A_1,C_1]^{(\rho)}=[1,n]^{(\rho)}$, as was the the case in $a_1^t+a_2^t$. Thus $(a_1+a_2)^t\ne a_1^t+a_2^t$, which implies  reducibility.
 
\end{enumerate} 

\item Let $B_2 \leq D_1$.

\begin{enumerate} 

\item Suppose $C_1<A_2.$

We again  illustrate   the situation graphically:
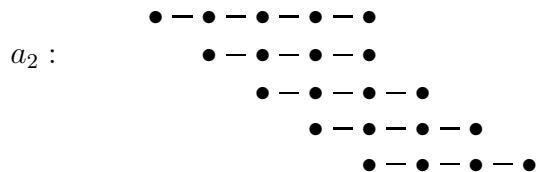
\begin{equation} 
\label{g44}
\xymatrix@C=.6pc@R=.1pc
{ 
a_1:& \bullet & \bullet \ar @{-}[l]  & \bullet \ar @{-}[l]  & \bullet \ar @{-}[l]  & \bullet \ar @{-}[l]  
\\ 
&&    \bullet & \bullet \ar @{-}[l]  & \bullet \ar @{-}[l]  & \bullet \ar @{-}[l]  & \bullet \ar @{-}[l]  
\\
a_2: & &  &  \bullet & \bullet \ar @{-}[l]  & \bullet \ar @{-}[l]  &  \bullet  \ar @{-}[l]   
  \\
  &&&&  \bullet & \bullet \ar @{-}[l]  & \bullet \ar @{-}[l]  &  \bullet  \ar @{-}[l]   
  \\
 & & &&&  \bullet & \bullet \ar @{-}[l]  & \bullet \ar @{-}[l]  &  \bullet  \ar @{-}[l]     
  \\
 && & &&&  \bullet & \bullet \ar @{-}[l]  & \bullet \ar @{-}[l]  &  \bullet \ar @{-}[l]     
} 
\end{equation}
In $a_1^t+a_2^t$ there is a unique segment which ends with exponent $D_2$. It is $[B_2,D_2]^{(\rho)}$.

We now start  MWA$^\la$. It starts with exponent $D_2$, and proceeds with $D_2-1$,\dots , $B_2$. Since $B_1<B_2$, we can find a segment in $a_2$  ending with $B_2-1$. Since $C_1<A_2$, this segments precedes the previous one ending with $B_2$. This implies that the segment in $(a_1+a_2)^t$ ending with exponent $D_2$ is not $[B_2,D_2]^{(\rho)}$. This implies $a_1^t+a_2^t\ne(a_1+a_2)^t$, which again implies  reducibility.

\item Suppose now $A_2 \leq C_1.$

We again  illustrate   the situation graphically:
\begin{equation} 
\label{g444}
\xymatrix@C=.6pc@R=.1pc
{ 
a_1:& \bullet & \bullet \ar @{-}[l]  & \bullet \ar @{-}[l]  & \bullet \ar @{-}[l]  & \bullet \ar @{-}[l]  
\\ 
&&    \bullet & \bullet \ar @{-}[l]  & \bullet \ar @{-}[l]  & \bullet \ar @{-}[l]  & \bullet \ar @{-}[l] 
\\
&&&    \bullet & \bullet \ar @{-}[l]  & \bullet \ar @{-}[l]  & \bullet \ar @{-}[l]  & \bullet \ar @{-}[l]  
\\
a_2: & &  &  \bullet & \bullet \ar @{-}[l]  & \bullet \ar @{-}[l]  &  \bullet  \ar @{-}[l]   
  \\
  &&&&  \bullet & \bullet \ar @{-}[l]  & \bullet \ar @{-}[l]  &  \bullet  \ar @{-}[l]   
  \\
 & & &&&  \bullet & \bullet \ar @{-}[l]  & \bullet \ar @{-}[l]  &  \bullet  \ar @{-}[l]     
  \\
 && & &&&  \bullet & \bullet \ar @{-}[l]  & \bullet \ar @{-}[l]  &  \bullet \ar @{-}[l]     
} 
\end{equation}
If $e\leq d$, then start MWA$^\la$, and the first segment that we get from the algorithm implies $a_1^t+a_2^t\ne(a_1+a_2)^t$, which implies  reducibility.

If $d\leq e$, then start $^\ra $MWA, and the first segment that we get from the algorithm implies $a_1^t+a_2^t\ne(a_1+a_2)^t$. Again we get  reducibility.
\end{enumerate} 

\end{enumerate} 
\end{proof}

\section{Irreducibility in the case of non-linked underlining sets of supports}
\label{non-linked}

We continue with the notation  of the last two sections.
In this section, we consider  the case where the underlying set of $\text{supp}(L(a(n,d)^{(\nu_{\rho'}^k\rho')})$ contains the underlying set of $\text{supp}(L(a(m,e)^{(\nu_{\rho'}^l\rho')})$. 
Our aim in this section is to prove that 
\begin{equation}
\label{irr}
L(a(n,d)^{(\nu_{\rho'}^k\rho')}) \t L(a(m,e)^{(\nu_{\rho'}^l\rho')})
\end{equation}
 is irreducible.
\subsection{Some remarks regarding irreducibility} We know by  criterion \eqref{BLM} that \eqref{irr} is irreducible if $L(a(n,d)^{(\nu_{\rho'}^k\rho')})$ and $L(a(m,e)^{(\nu_{\rho'}^l\rho')})$ are not crossed.

Since $L(a(n,d)^{(\nu_{\rho'}^k\rho')}) \t L(a(m,e)^{(\nu_{\rho'}^l\rho')})$ is irreducible if and only the dual representation   $L(a(d,n)^{(\nu_{\rho'}^k\rho')}) \t L(a(e,m)^{(\nu_{\rho'}^l\rho')})$ is irreducible, it is enough to consider the case
$$
n\leq d.
$$

Below we use   the notation $\pi_i$, $a_i$, $A_i,B_i,C_i,D_i$, $i=1,2$, from the previous section.

Now the condition of inclusion of underlying sets tells us in this notation,
$$
A_1\leq A_2,
$$ 
$$
D_2\leq D_1,
$$
while the assumption $n\leq d$ becomes
$$
C_1\leq B_1.
$$

\subsection{On crossed case} 
Continuing the above analysis, we know that the segments are crossed if and only if both conditions below hold:
\begin{enumerate}

\item \hskip20mm
$C_2+1\leq C_1$ or 
$
[B_1+1,D_1+1]\cap [A_2,B_2]\ne \emptyset;
$

\vskip2mm

\item
\hskip20mm
$B_2+1\leq B_1$ or 
$
[C_1+1,D_1+1]\cap [A_2,C_2]\ne \emptyset.
$

\end{enumerate}
Since we know that \eqref{irr} is irreducible if $a_1$ and $a_2$ are not crossed,
we  shall analyze the case of two crossed representations in this subsection.Therefore, we  assume that both above conditions hold in the rest of this subsection.
 Since $B_2,C_2<D_1+1$, the above requirements are equivalent to 
 \begin{enumerate}

\item \hskip20mm
$C_2+1\leq C_1$ or 
$
B_1+1 \leq B_2
$
(i.e., $C_2< C_1$ or 
$
B_1< B_2
$);

\vskip2mm

\item
\hskip20mm
$B_2+1\leq B_1$ or 
$
C_1+1\leq  C_2
$
(i.e., $B_2< B_1$ or 
$
C_1<  C_2
$).

\end{enumerate}

Thus the crossing condition is
$$
C_2<C_1\text{ and } B_2<B_1 \text{ \ \ or\ \ }C_1<C_2\text{ and } B_1<B_2.
$$

In studying  the question of irreducibility of \eqref{irr}, without lost of generality we can  assume
$$
A_2+D_2\leq  A_1+D_1
$$
 (if this is not the case, passing to contragredients will bring us to this case).  
 This implies 
 $$
 B_2+C_2\leq B_1+ C_1.
 $$
 
 Therefore, the crossing condition for the case $A_2+D_2\leq A_1+D_1$ is
 $$
 C_2<C_1\text{ and } B_2<B_1.
 $$
 Observe that $B_2<B_1$ (i.e. $b+e-1< d$) implies
 $$
 e< d.
 $$

\subsection{Additivity of $^t$ in the linked case} 

In the lemma below we only assume that the underlying  set of the cuspidal support of $a_1$ contains the corresponding set of $a_2$. We continue with the previous notation.

\begin{lemma} 
With the above  notation and assumptions, we have
$$
a_1^t+a_2^t
=(a_1+a_2)^t.
$$ 
\end{lemma}

\begin{proof} We shall first list some simple reductions of the proof of the lemma.

\begin{enumerate}
\item Obviously, it is enough to prove the claim of the lemma for   $(a_1,a_2)$ or $(a_1^t,a_2^t)$. 

\item The above observation and Theorem 4.2 of \cite{Z} imply that the lemma holds if  $n=1$ or $d=1$. Therefore, it is enough to prove the lemma in the case $n\geq 2$ and $d\geq 2$.

\item The claim of the lemma holds if $A_1=A_2$ and $D_1=D_2$ (since in this case we can twist to the unitarizable setting, in which case Corollary 1 of \cite{Ba-Sp} implies irreducibility of $L(a_1)\t (a_2)$, and then use the fact that irreducibility of $L(a_1)\t (a_2)$ implies the claim of the lemma). Therefore, it is enough to prove the lemma in the case $A_1<A_2$ or $D_2>D_1$.

\item Suppose that we are not in the case of $A_1=A_2$ and $D_1=D_2$ (when we know that the claim of the lemma holds). Then it is enough to prove the lemma in the case $D_2>D_1$ (passing to the herimitian contragredient and twisting by a character will bring us to the case $A_1<A_2$). 

\end{enumerate}

Now we shall prove the lemma by induction. The above reduction (2)  provides the basis of induction. Fix some $n\geq 2$ and $d\geq 2$ and suppose  that the claim of the lemma holds for pairs $n',d'$ where $n'<n$ or $d'<d$. By reduction (4), it is enough to consider the case $D_2>D_1$.

We start MWA$\la$. We  must begin with $D_1$. An easy discussion related to the fact if the segments in $a_2$ are longer, equal or shorter then the ones in $a_1$, implies that the first step of the algorithm will produce segment $[B_1,D_1]^{(\rho)}$ (see Remark \ref{rem-alg}).
We illustrate the situation by the drawing below:
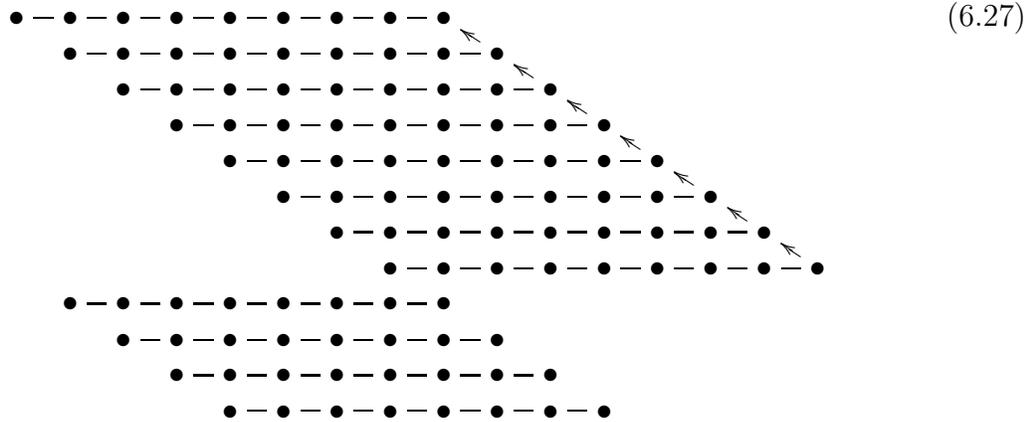
\begin{equation}  
\label{g7}
\xymatrix@C=.6pc@R=.1pc
{ 
   \bullet  & \bullet  \ar @{-}[l] & \bullet \ar @{-}[l]  &  \bullet \ar @{-}[l]   & \bullet  \ar @{-}[l] & \bullet \ar @{-}[l]  &\bullet \ar @{-}[l]  & \bullet \ar @{-}[l] & \bullet {\ar @{-}[l]} 
\\ 
& \bullet   & \bullet \ar @{-}[l]  & \bullet \ar @{-}[l]  
& \bullet \ar @{-}[l]  &  \bullet \ar @{-}[l]   & \bullet  \ar @{-}[l] & \bullet \ar @{-}[l] & \bullet \ar @{-}[l] & \bullet {\ar @{-}[l] } \ar[ul]
\\
 && \bullet   & \bullet \ar @{-}[l]  & \bullet \ar @{-}[l]  & \bullet \ar @{-}[l]  &  \bullet \ar @{-}[l]   & \bullet  \ar @{-}[l] & \bullet \ar @{-}[l] & \bullet \ar @{-}[l] & \bullet {\ar @{-}[l] } \ar[ul]
\\
    &&& \bullet   & \bullet \ar @{-}[l]  & \bullet \ar @{-}[l]   & \bullet \ar @{-}[l]  &  \bullet \ar @{-}[l]   & \bullet  \ar @{-}[l] & \bullet \ar @{-}[l] & \bullet \ar @{-}[l] & \bullet {\ar @{-}[l] } \ar[ul]
  \\
 & & && \bullet   & \bullet \ar @{-}[l]  & \bullet \ar @{-}[l]  & \bullet \ar @{-}[l]  &  \bullet \ar @{-}[l]   & \bullet  \ar @{-}[l] & \bullet \ar @{-}[l]  & \bullet \ar @{-}[l]  & \bullet {\ar @{-}[l] } \ar[ul]
  \\
 && & && \bullet   & \bullet \ar @{-}[l]  & \bullet \ar @{-}[l]  & \bullet \ar @{-}[l]  &  \bullet \ar @{-}[l]   & \bullet  \ar @{-}[l] & \bullet \ar @{-}[l] & \bullet \ar @{-}[l]  & \bullet {\ar @{-}[l] } \ar[ul]
  \\
& && & && \bullet   & \bullet \ar @{-}[l]  & \bullet \ar @{-}[l]  & \bullet \ar @{-}[l]  &  \bullet \ar @{-}[l]   & \bullet  \ar @{-}[l] & \bullet \ar @{-}[l] & \bullet \ar @{-}[l]  & \bullet {\ar @{-}[l] } \ar[ul]
\\
&& && & && \bullet   & \bullet \ar @{-}[l]  & \bullet \ar @{-}[l]  & \bullet \ar @{-}[l]  &  \bullet \ar @{-}[l]   & \bullet  \ar @{-}[l] & \bullet \ar @{-}[l] & \bullet \ar @{-}[l]  & \bullet {\ar @{-}[l] } \ar[ul]
\\
    & \bullet & \bullet \ar @{-}[l] & \bullet \ar @{-}[l] & \bullet \ar @{-}[l] & \bullet \ar @{-}[l] & \bullet \ar @{-}[l] & \bullet \ar @{-}[l]  & \bullet \ar @{-}[l] 
\\ 
 && \bullet & \bullet \ar @{-}[l]  & \bullet \ar @{-}[l] & \bullet \ar @{-}[l] & \bullet \ar @{-}[l] & \bullet \ar @{-}[l] & \bullet \ar @{-}[l]  & \bullet \ar @{-}[l] 
\\
 &&& \bullet & \bullet \ar @{-}[l]   & \bullet \ar @{-}[l] & \bullet \ar @{-}[l] & \bullet \ar @{-}[l] & \bullet \ar @{-}[l] & \bullet \ar @{-}[l]  & \bullet \ar @{-}[l] 
 \\
 &&&& \bullet & \bullet \ar @{-}[l]  & \bullet \ar @{-}[l] & \bullet \ar @{-}[l] & \bullet \ar @{-}[l] & \bullet \ar @{-}[l] & \bullet \ar @{-}[l]  & \bullet \ar @{-}[l] 
} 
\end{equation}
Denote
$$
a_1'=a(n,d-1)^{(\nu_{\rho'}^{-1/2}\nu_{\rho'}^k\rho')}.
$$
Now we see that $(a_1+a_2)^\la=a_1'+a_2$. Further,  the inductive assumption gives $(a_1'+a_2)^t=(a_1')^t+a_2^t$.
From this follows that $(a_1+a_2)^t=([B_1,D_1]^{(\rho)}) +(a_1'+a_2)^t= ([B_1,D_1]^{(\rho)}) +(a_1')^t+a_2^t=([B_1,D_1]^{(\rho)}) +a(d-1,n)^{(\nu_{\rho'}^{-1/2}\nu_{\rho'}^k\rho')}+a_2^t=a(d,n)^{(\nu_{\rho'}^k\rho')}+a_2^t=a_1^t+a_2^t.$ This completes the proof of the lemma.
\end{proof}

\subsection{Proof of the  irreducibility in the linked case}

\begin{lemma}
If $c< a_2$ then $c^t \not< a_2^t$.
\end{lemma} 

\begin{proof}  Suppose that there exists $c\in M(\SC)$ such that $c< a_2$ and $c^t < a_2^t$. 

All the segments in $a_2^t$ have length $m$, and any linking among segments of $a_2^t$ will produce a segment  longer than $m$. Further, linkings may only increase the maximal length of segments. Since $c^t < a_2^t$, we have   a segment in $c$ of length at least $m+1$.

The ends of the segments from $a_2$ form the segment $[B_2,D_2]^{(\rho)}$. Therefore, the ends of segments from $c$ are contained in $[B_2,D_2]$. By \ref{est}, this implies  that the segments in $c^t$ are not longer than $D_2-B_2+1=C_2-A_2+1=m$.

This contradiction completes the proof of the lemma.
\end{proof}

\begin{proposition}
Suppose that underlying set of $\text{supp}(L(a(n,d)^{(\nu_{\rho'}^k\rho')})$ contains the underlying set of $\text{supp}(L(a(m,e)^{(\nu_{\rho'}^l\rho')})$. Then $L(a(n,d)^{(\nu_{\rho'}^k\rho')}) \t L(a(m,e)^{(\nu_{\rho'}^l\rho')})$ is irreducible.
\end{proposition}

\begin{proof} In is enough to prove the proposition in the crossed case. We shall assume this, and 
 continue with the previous notation and the previous assumptions: 
$$
B_2<B_1,\quad C_2<C_1,\quad C_1\leq B_1 \quad \text{ (i.e., }n\leq d)
$$
(see 6.2). We illustrate $a_1$ and $a_2$   by the following drawing:
\begin{equation} 
\label{g11}
\xymatrix@C=.6pc@R=.1pc
{ 
   &\bullet  & \bullet  \ar @{-}[l]  & \bullet \ar @{-}[l]  &  \bullet \ar @{-}[l]   & \bullet  \ar @{-}[l]  & \bullet \ar @{-}[l]  & \bullet \ar @{-}[l]   & \bullet \ar @{-}[l]   & \bullet \ar @{-}[l] 
\\ 
&& \bullet   & \bullet \ar @{-}[l]  & \bullet \ar @{-}[l]  
& \bullet \ar @{-}[l]  &  \bullet \ar @{-}[l]   & \bullet  \ar @{-}[l] & \bullet \ar @{-}[l]   & \bullet \ar @{-}[l]   & \bullet \ar @{-}[l] 
\\
& && \bullet   & \bullet \ar @{-}[l]  & \bullet \ar @{-}[l]  & \bullet \ar @{-}[l]  &  \bullet \ar @{-}[l]   & \bullet  \ar @{-}[l] & \bullet \ar @{-}[l]   & \bullet \ar @{-}[l]   & \bullet \ar @{-}[l] 
\\
   & &&& \bullet   & \bullet \ar @{-}[l]  & \bullet \ar @{-}[l]   & \bullet \ar @{-}[l]  &  \bullet \ar @{-}[l]   & \bullet  \ar @{-}[l] & \bullet \ar @{-}[l]  & \bullet \ar @{-}[l]   & \bullet \ar @{-}[l] 
  \\
& & & && \bullet   & \bullet \ar @{-}[l]  & \bullet \ar @{-}[l]  & \bullet \ar @{-}[l]  &  \bullet \ar @{-}[l]   & \bullet  \ar @{-}[l] & \bullet \ar @{-}[l]  & \bullet \ar @{-}[l]   & \bullet \ar @{-}[l] 
  \\
& && & && \bullet   & \bullet \ar @{-}[l]  & \bullet \ar @{-}[l]  & \bullet \ar @{-}[l]  &  \bullet \ar @{-}[l]   & \bullet  \ar @{-}[l]   & \bullet \ar @{-}[l]  & \bullet \ar @{-}[l]   & \bullet \ar @{-}[l] 
  \\
&& && & && \bullet   & \bullet \ar @{-}[l]  & \bullet \ar @{-}[l]  & \bullet \ar @{-}[l]  &  \bullet \ar @{-}[l]   & \bullet  \ar @{-}[l]   & \bullet \ar @{-}[l]  & \bullet \ar @{-}[l]   & \bullet \ar @{-}[l] 
\\
 &     &   \bullet &\bullet  \ar @{-}[l] & \bullet  \ar @{-}[l]    & \bullet \ar @{-}[l] 
\\
 &    &  &   \bullet &\bullet  \ar @{-}[l] & \bullet  \ar @{-}[l]    & \bullet \ar @{-}[l] 
\\
 &&&& \bullet   & \bullet  \ar @{-}[l]  & \bullet \ar @{-}[l]   & \bullet \ar @{-}[l] 
 \\
 &&&&  & \bullet& \bullet  \ar @{-}[l]  & \bullet \ar @{-}[l]   & \bullet \ar @{-}[l] 
 \\
  &&&&&&\bullet   & \bullet  \ar @{-}[l]  & \bullet \ar @{-}[l]   & \bullet \ar @{-}[l] 
} 
\end{equation}

Suppose that 
$$
L(a_1)\t L(a_2)
$$
 reduces. Then there exists 
\begin{equation}
\label{<}
f< a_1+a_2
\end{equation}
such that 
$L(f)$ is a subquotient of $ L(a_1)\t L(a_2)$. Therefore
$$
L(f^t)\text{\ is a subquotient of \ } L(a_1^t)\t L(a_2^t).
$$
Thus, $f^t\leq a_1^t+a_2^t$. Observe that $f^t= a_1^t+a_2^t$ would contradict Lemma 6.1 since $f\ne  a_1+a_2$. Thus
\begin{equation}
\label{<t}
f^t< a_1^t+a_2^t.
\end{equation}

In the rest of the proof we fix some
$
f<a_1+a_2
$
satisfying
$
f^t< a_1^t+a_2^t.
$

Since $C_2<C_1$, the beginnings of all segments in $a_1+a_2$ are contained in $[A_1,C_1]^{(\rho)}$. This also holds for $f$ since $f<a_1+a_2$  (since the beginnings of   all segments in $f$ are   contained in the beginnings of all segments in $a_1+a_2$). This implies that lengths of segments in $f^t$ are at most $n$. Since $f^t< a_1^t+a_2^t$, we conclude that no segment from $a_1^t$ can take part in any linking which produces $f^t$. Therefore
$$
f^t=a_1^t+c_t
$$
for some $c_t\in M(\SC)$. Observe that above considerations imply
$$
c_t<a_2^t.
$$

We  now apply the above arguments to $f^t$ instead of $f$, and use the fact  that $B_2<B_1$. In the same way we get for $f=(f^t)^t$ that
$$
f=a_1+c
$$
for some
$$
c<a_2.
$$

We shall now compute $(a_1+c)^t$, i.e. $f^t$. We consider two cases.

The first case is 
$$
D_2< C_1.
$$
Then the first $d$ steps of the algorithm produce segments
$$
[B_1,D_1]^{(\rho)}, [B_1-1,D_1-1]^{(\rho)},\dots, [A_1,C_1]^{(\rho)}.
$$
The multisegment formed by these segments is exactly $a_1^t$. The rest of the algorithm  MWA$^\la$ gives obviously $c^t$. Therefore, 
$$
(a_1+c)^t=a_1^t+c^t.
$$

We shall now show that the same formula holds also in the remaining case
$$
C_1\leq D_2.
$$
The first $D_1-D_2+1$ steps of the algorithm produce segments
$$
[B_1,D_1]^{(\rho)}, [B_1-1,D_1-1]^{(\rho)},\dots, [D_1-D_2+1+B_1-D_1,D_1-D_2+1]^{(\rho)}
$$
(for the last step there were two ends to start MWA$^\la$, one coming from the end of a truncated segment from  $a_1$, and the other from the end of a segment from $a_2$; the condition $C_2<C_1$ implies that we must start with 
the end of a truncated segment from $a_1$).

We illustrate this stage of the algorithm    by the following drawing:
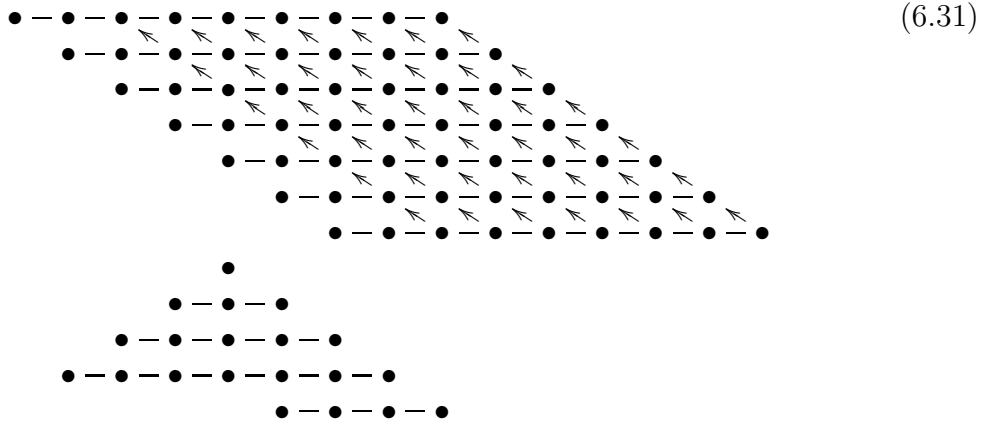
\begin{equation} 
\label{g111}
\xymatrix@C=.6pc@R=.1pc
{ 
   &\bullet  & \bullet  {\ar @{-}[l]}  & \bullet  {\ar @{-}[l]}   & \bullet {\ar @{-}[l]}  &  \bullet {\ar @{-}[l]}   & \bullet  {\ar @{-}[l]}  & \bullet {\ar @{-}[l]}  & \bullet {\ar @{-}[l]}   & \bullet {\ar @{-}[l]} 
\\ 
&& \bullet  & \bullet  {\ar @{-}[l]}   & \bullet {\ar @{-}[l]} \ar[lu]  & \bullet {\ar @{-}[l]} \ar[lu]  
& \bullet {\ar @{-}[l]} \ar[lu]  &  \bullet {\ar @{-}[l]} \ar[lu]   & \bullet  {\ar @{-}[l]} \ar[lu] & \bullet {\ar @{-}[l]} \ar[lu]   & \bullet {\ar @{-}[l]} \ar[lu] 
\\
& && \bullet  & \bullet  {\ar @{-}[l]}   & \bullet {\ar @{-}[l]} \ar[lu]  & \bullet {\ar @{-}[l]} \ar[lu]  & \bullet {\ar @{-}[l]} \ar[lu]  &  \bullet {\ar @{-}[l]} \ar[lu]   & \bullet  {\ar @{-}[l]} \ar[lu] & \bullet {\ar @{-}[l]} \ar[lu]   & \bullet {\ar @{-}[l]} \ar[lu] 
\\
   & &&& \bullet   & \bullet  {\ar @{-}[l]}  & \bullet {\ar @{-}[l]} \ar[lu]  & \bullet {\ar @{-}[l]} \ar[lu]   & \bullet {\ar @{-}[l]} \ar[lu]  &  \bullet {\ar @{-}[l]} \ar[lu]   & \bullet  {\ar @{-}[l]} \ar[lu] & \bullet {\ar @{-}[l]} \ar[lu]  & \bullet {\ar @{-}[l]} \ar[lu] 
  \\
& & & && \bullet   & \bullet  {\ar @{-}[l]}  & \bullet {\ar @{-}[l]} \ar[lu]  & \bullet {\ar @{-}[l]} \ar[lu]  & \bullet {\ar @{-}[l]} \ar[lu]  &  \bullet {\ar @{-}[l]} \ar[lu]   & \bullet  {\ar @{-}[l]} \ar[lu] & \bullet {\ar @{-}[l]} \ar[lu]  & \bullet {\ar @{-}[l]} \ar[lu] 
  \\
& && & && \bullet   & \bullet  {\ar @{-}[l]}  & \bullet {\ar @{-}[l]} \ar[lu]  & \bullet {\ar @{-}[l]} \ar[lu]  & \bullet {\ar @{-}[l]} \ar[lu]  &  \bullet {\ar @{-}[l]} \ar[lu]   & \bullet  {\ar @{-}[l]} \ar[lu]   & \bullet {\ar @{-}[l]} \ar[lu]  & \bullet {\ar @{-}[l]} \ar[lu] 
  \\
&& && & && \bullet   & \bullet  {\ar @{-}[l]}  & \bullet {\ar @{-}[l]} \ar[lu]   & \bullet {\ar @{-}[l]} \ar[lu]  & \bullet {\ar @{-}[l]} \ar[lu]  &  \bullet {\ar @{-}[l]} \ar[lu]   & \bullet  {\ar @{-}[l]} \ar[lu]   & \bullet {\ar @{-}[l]} \ar[lu]  & \bullet {\ar @{-}[l]} \ar[lu] 
\\
 &   &&&  &    \bullet
\\
 &  &  &  &   \bullet  & \bullet  \ar @{-}[l]    & \bullet \ar @{-}[l] 
\\
 &&& \bullet &\bullet  \ar @{-}[l]   & \bullet  \ar @{-}[l]  & \bullet \ar @{-}[l]   & \bullet \ar @{-}[l] 
 \\
 && \bullet &\bullet  \ar @{-}[l] & \bullet  \ar @{-}[l] & \bullet \ar @{-}[l]& \bullet  \ar @{-}[l]  & \bullet \ar @{-}[l]   & \bullet \ar @{-}[l] 
 \\
  &&&&&&\bullet    & \bullet  \ar @{-}[l]  & \bullet \ar @{-}[l]   & \bullet \ar @{-}[l] 
} 
\end{equation}

In our example, after removing points of the segments that we have already used,  we are now in the following situation:
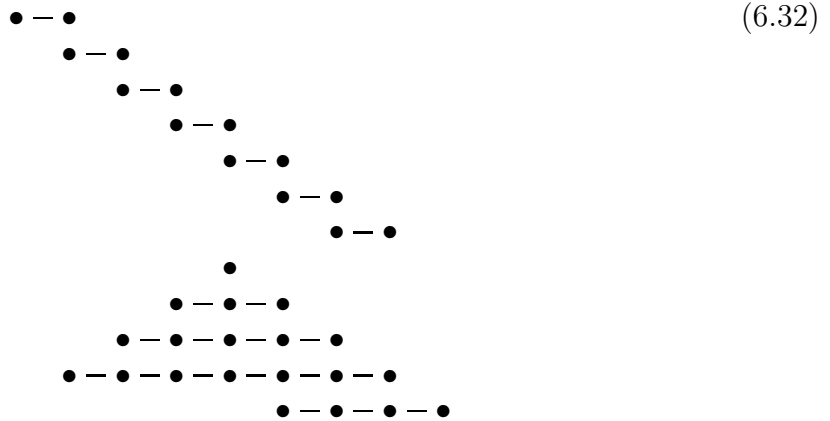
\begin{equation} 
\label{g1111}
\xymatrix@C=.6pc@R=.1pc
{ 
   &\bullet   & \bullet  \ar @{-}[l]  
\\ 
&& \bullet    & \bullet  \ar @{-}[l]  
\\
& && \bullet    & \bullet  \ar @{-}[l]  
\\
   & &&& \bullet    & \bullet  \ar @{-}[l]  
     \\
& & & && \bullet    & \bullet  \ar @{-}[l]  
  \\
& && & && \bullet    & \bullet  \ar @{-}[l]  
  \\
&& && & && \bullet    & \bullet  \ar @{-}[l]  
\\
 &     &   &&  & \bullet  
\\
 &  &  &  &   \bullet  & \bullet  \ar @{-}[l]   & \bullet  \ar @{-}[l]  
\\
 &&& \bullet &\bullet  \ar @{-}[l]   & \bullet  \ar @{-}[l]  & \bullet \ar @{-}[l]  & \bullet  \ar @{-}[l]  
 \\
 && \bullet &\bullet  \ar @{-}[l] & \bullet  \ar @{-}[l] & \bullet \ar @{-}[l]& \bullet  \ar @{-}[l]  & \bullet \ar @{-}[l]  & \bullet  \ar @{-}[l]  
 \\
  &&&&&&\bullet    & \bullet  \ar @{-}[l]  & \bullet \ar @{-}[l]  & \bullet  \ar @{-}[l]  
} 
\end{equation}

We must start the next step  with the end $D_2$ of the segment  from $c$. Denote this segment by $[X,D_2]$. This segment is longer than $[C_1,D_2]$ since $X\leq C_2<C_1$. Therefore, in all further stages of MWA$^\la$ starting with this point, we will have segments of the same length as $[X,D_2]$, or longer. Therefore, at this step, only the segments from $c$ will take part (see Remark \ref{rem-alg}).

Now we go to $D_2-1$. Then we need to start with a segment coming from $a_1$, since it is shorter. In the same way as before,  we now complete this step with segments coming from $a_1$.

We  continue these steps with possible beginnings  $D_2-1, D_2-2$, \dots, $C_1$ (possibly several times with each of them). In the same way as above, at each step we shall deal either with segments coming from $a_1$, or with segments coming from $c$ (if we start with the end of a segment coming from $a_1$, this is clear; for the end of a segment coming from $c$, $C_2<C_1$ implies that we must complete such step with segments coming from $c$).

Further, after $C_1$, the remaining steps take part only inside segments coming from $c$.

The above discussion implies that final result of the algorithm in this case is also 
$$
(a_1+c)^t=a_1^t+c^t
$$
because each step of the algorithm can be performed  using the segments coming either entirely from $a_1$ or entirely from $c$.

Now we can complete the proof.  We have just proved  $f^t=a_1^t+c^t$. Recall that $f^t=a_1^t+c_t$. This implies 
$$
c_t=c^t.
$$
 Further, we know from earlier
$$
c<a_2
$$
 and 
 $
 c_t<a_2^t.
 $
  Now $c_t=c^t$ implies
 $$
 c^t<a_2^t.
 $$
The preceding lemma now  implies a contradiction. 
The proof of Proposition 6.4 is now  complete.
\end{proof}

\begin{remark} 
A. M\'inguez has explained us  how to  get in a simple way the above proposition from Lemma 1.2 of \cite{BLM}  (providing in this way a proof in terms of Jacquet modules). The case when one representation is cuspidal follows from \cite{M}. The general case follows by induction. For the induction, one needs to check the multiplicity one of the inducing representation in the  Jacquet module of the induced representation, to be able to apply Lemma 1.2 of \cite{BLM}.
\end{remark}
\label{rem-Min}

\section{Reducibility criterion}

\subsection{The case of two essentially Speh representations}

Propositions 5.1 and 6.3 imply the following 

\begin{theorem}
Let $\pi_1=u_{ess}(\d_1,m_1)$ and $\pi_2=u_{ess}(\d_2,m_2)$ for some $\d_1,\d_2\in \C D$ and some positive integers $m_1,m_2$. Then
$$
\pi_1\t\pi_2
$$
 is reducible if and only if the underlying sets of the cuspidal supports of $\pi_1$ and $\pi_2$ are  linked and $\pi_1$ and $\pi_2$ are  crossed.
\end{theorem}

We now  explain when the reducibility happens in a different way. 

First, to have reducibility, we need to have both  representations $\pi_1$ and $\pi_2$ supported on the same cuspidal $\Z$-line, i.e., in $\{\nu^k_\rho \rho; k\in \Z\}$ for some $k\in \C C$. We  assume this in what follows.

Write
\begin{equation}
\label{1}
\pi_i=L(([A_i,B_i],[A_i+1,B_i+1],\dots,[C_i,D_i])^{(\rho)}), \ i=1,2.
\end{equation}
We can always chose $\rho$ so that $A_i\in \Z$.
Here
$$
A_i\leq B_i,C_i \leq D_i
$$
and
$$
A_i+D_i=B_i+C_i
$$
for $i=1,2$. 
We use the shorthand
$$
\pi_i=u_{ess}
\left(
\begin{matrix}
A_i  \hskip3mm & B_i  \hskip3mm
\\
    \hskip3mm C_i&    \hskip3mm D_i
\end{matrix}
 \right)^{(\rho)}
$$

To have reducibility, the linking condition for the underlying sets of the cuspidal supports must be satisfied. This implies $A_1\ne A_2$. Without lost of generality we can assume
\begin{equation}
\label{2}
A_1<A_2.
\end{equation}
Now, the linking condition is equivalent to
\begin{equation}
\label{3}
D_1<D_2
\end{equation}
and
\begin{equation}
\label{4}
A_2\leq D_1+1.
\end{equation}

The crossing condition is now equivalent to
$$
B_1<B_2 \text{ and } C_1< C_2.
$$

Write 
$$
\left(
\begin{matrix}
A \hskip3mm & B  \hskip3mm
\\
    \hskip3mm C&    \hskip3mm D
\end{matrix}
 \right)
 <_{strong}
 \left(
\begin{matrix}
A_1  \hskip3mm & B_1  \hskip3mm
\\
    \hskip3mm C_1&    \hskip3mm D_1
\end{matrix}
 \right)
 $$
 if $A<A_1,B<B_1,C<C_1$ and $D<D_1$. 
 
 Let us summarize: If  $\pi_1$ and $\pi_2$ are not supported on the same cuspidal $\Z$-lines, then $\pi_1\t\pi_2$ is irreducible.
If they are supported by the same cuspidal $\Z$-line, then we can write them as
$$
\pi_i=u_{ess}
\left(
\begin{matrix}
A_i  \hskip3mm & B_i  \hskip3mm
\\
    \hskip3mm C_i&    \hskip3mm D_i
\end{matrix}
 \right)^{(\rho)}
$$
where $A_i,B_i,D_i,D_i\in \Z$, 
$$
A_i\leq B_i,C_i \leq D_i
$$
and
$$
A_i+D_i=B_i+C_i
$$
for $i=1,2$.

\begin{theorem} The representation 
\begin{equation}
\label{repr}
\pi_1\t \pi_2=u_{ess}
\left(
\begin{matrix}
A_1  \hskip3mm & B_1 \hskip3mm
\\
    \hskip3mm C_1 &    \hskip3mm D_1
\end{matrix}
 \right)^{(\rho)}
 \t
 u_{ess}
\left(
\begin{matrix}
A_2  \hskip3mm & B_2  \hskip3mm
\\
    \hskip3mm C_2 &    \hskip3mm D_2
\end{matrix}
 \right)^{(\rho)}
 \end{equation}
   reduces  if and only if 
  $$
  [A_1,D_1]_\Z\cup[A_2,D_2]_\Z  
$$
is a $\Z$-segment, and
$$
\left(
\begin{matrix}
A_1  \hskip3mm & B_1  \hskip3mm
\\
    \hskip3mm C_1&    \hskip3mm D_1
\end{matrix}
 \right)
 <_{strong}
 \left(
\begin{matrix}
A_2  \hskip3mm & B_2  \hskip3mm
\\
    \hskip3mm C_2&    \hskip3mm D_2
\end{matrix}
 \right)
\text{ \ or \ }
\left(
\begin{matrix}
A_2  \hskip3mm & B_2  \hskip3mm
\\
    \hskip3mm C_2&    \hskip3mm D_2
\end{matrix}
 \right)
  <_{strong}
\left(
\begin{matrix}
A_1  \hskip3mm & B_1  \hskip3mm
\\
    \hskip3mm C_1&    \hskip3mm D_1
\end{matrix}
 \right)
.
$$
\end{theorem}

\subsection{General case} The following technical lemma follows very easily from  Theorem 7.2.

\begin{lemma}
Let the representations
$$
\pi_i=u_{ess}
\left(
\begin{matrix}
A_i  \hskip3mm & B_i \hskip3mm
\\
    \hskip3mm C_i &    \hskip3mm D_i
\end{matrix}
 \right)^{(\rho)},\qquad i=1,2,
 $$
 satisfy:
 \begin{enumerate}

\item $\pi_1\t\pi_2$ is irreducible;

\item $C_2+D_2\leq C_1+D_1$;

\item if $C_2+D_2= C_1+D_1$, then $1\leq D_1-B_1\leq D_2-B_2$.
 
 \end{enumerate}
 Denote 
 $$
 C_1'=C_1-1, \qquad D_1'=D_1-1 \text{\qquad and \qquad}\pi_1'=
 u_{ess}
\left(
\begin{matrix}
A_1  \hskip3mm & B_1 \hskip3mm
\\
    \hskip3mm C_1' &    \hskip3mm D_1'
\end{matrix}
 \right)^{(\rho)}
 $$
 Then
 \begin{equation}
 \label{-1}
 \pi_1'\t
 \pi_2
 \end{equation}
 is irreducible.

\end{lemma}

\begin{proof} If $[A_2,D_2]\cup [A_1,D_1-1]$ is not a $\Z$-segment, then directly follows that \eqref{-1} is irreducible. Therefore, we suppose that $[A_2,D_2]\cup [A_1,D_1-1]$ is a segment. Then clearly, $[A_2,D_2]\cup [A_1,D_1]$ is a $\Z$-segment. 

Denote
$$
M_i=\left(
\begin{matrix}
A_i  \hskip3mm & B_i \hskip3mm
\\
    \hskip3mm C_i &    \hskip3mm D_i
\end{matrix}
 \right), \ \ i=1,2,
 \text{\qquad and \qquad}
 M_1'=
 \left(
\begin{matrix}
A_1  \hskip3mm & B_1 \hskip3mm
\\
    \hskip3mm C_1' &    \hskip3mm D_1'
\end{matrix}
 \right).
$$
The condition (1) of the lemma and the previous theorem imply
\begin{equation}
\label{ir}
M_1
 \not<_{strong}
 M_2
\text{ \ and \ }
M_2
 \not <_{strong}
M_1.
  \end{equation}
  Suppose that \eqref{-1} is reducible. Then the above theorem  imples
  $$
  M_1
 <_{strong}
 M_2'
\text{ \ or \ }
M_2'
  <_{strong}
M_1.
  $$
Suppose
$ M_2
 <_{strong}
 M_1'$. Then 
 $$
 A_2<A_1,\quad B_2<B_1,\quad C_2<C_1-1,\quad D_2<D_1-1.
 $$
 Then obviously $C_2<C_1$ and $D_2<D_1$, which implies $M_2<_{strong} M_1$. This  contradicts  \eqref{ir}. 

Therefore
$ M_1'
 <_{strong}
 M_2$. This implies
$
 A_1<A_2, B_1<B_2, C_1-1<C_2, D_1-1<D_2.
$
 i.e.
 \begin{equation}
\label{nejedakost}
 A_1<A_2,\quad B_1<B_2,\quad C_1\leq C_2,\quad D_1\leq D_2.
 \end{equation}
 Now the last two inequalities and condition (2) from the lemma imply that
$$
C_1=C_2\text{\ \ and \ \ } D_1=D_2.
$$
Now $B_1<B_2$ implies 
\begin{equation}
\label{pomocno}
D_2-B_2<D_1-B_1
\end{equation}
 (since $D_1=D_2$).

From the other side,  $C_1=C_2$ and $D_1=D_2$ yield $C_1+D_1=C_2+D_2$. Because of this, we can apply the condition (3) of the lemma, which says 
$$
1\leq D_1-B_1\leq D_2-B_2.
$$
 This contradicts \eqref{pomocno}. The proof is now complete.
 \end{proof}

The claim of the following lemma is essentially contained in I.9 of \cite{MoeW-GL} (the proof in \cite{MoeW-GL} is based on section I.8 there). The proof that we present bellow is very elementary (it uses the strategy of proof of Proposition 8.5 from \cite{Z}, which was also used in \cite{MoeW-GL}).

\begin{lemma} Let $\pi_i=L(d_i), i=1,\dots,k$, be essentially Speh representations such that
$$
\pi_i\t\pi_j
$$
is irreduciblele for all (different) $i,j\in\{1,\dots,k\}$. Let $\s$ be a permutation of $\{1,\dots,k\}$. Then 
\begin{enumerate}

\item $\pi_1\t\pi_2\t\dots\t\pi_k\cong \pi_{\s(1)}\t\pi_{\s(2)}\t\dots\t\pi_{\s(k)};$

\item $\pi_1\t\pi_2\t\dots\t\pi_k$ is a quotient of $\lambda(d_1+\dots+d_k)$;

\item $\pi_1\t\pi_2\t\dots\t\pi_k$ has unique irreducible quotient. The irreducible quotient is isomorphic to $L(d_1+\dots+d_k)$. It has multiplicity one in $\pi_1\t\pi_2\t\dots\t\pi_k$.

\end{enumerate}
\end{lemma}

\begin{proof} The claim (1) follows  directly from the basic property \eqref{ass} of $\t$ (and the fact that each permutation is a product of transpositions). Further, (3) is a direct consequence of (2). Therefore, it remains to prove (2). We shall do this by induction with respect to cardinality of the multi set $d_1+\dots+d_k$.

Observe that for $k\leq 2$, $\pi_1\t\pi_2\t\dots\t\pi_k$ is irreducible, which implies $\pi_1\t\pi_2\t\dots\t\pi_k=L(d_1+\dots+d_k)$. Therefore, obviously (3) holds in this situation. This provides the basis for the induction.

Fix now $m\geq 3$ and let $\pi_i=L(d_i), i=1,\dots,k$, be essentially Speh representations such that
$
\pi_i\t\pi_j
$
is irreduciblele for all (different) $i,j\in\{1,\dots,k\}$ and that cardinality of $d_1+\dots+d_k$ is $m$. Suppose that (2) holds in the case of cardinality $m-1$.

First, it is enough to consider the case of $k\geq3$. Further, it is enough to consider the case when we can write all $\pi_i$ as
$$
\pi_i=u_{ess}
\left(
\begin{matrix}
A_i  \hskip3mm & B_i \hskip3mm
\\
    \hskip3mm C_i &    \hskip3mm D_i
\end{matrix}
 \right)^{(\rho)},\qquad i=1,2,\dots,k.
 $$
 Consider all indexes $j$ with maximal $C_j+D_j$, and among them,  chose an index with minimal  $D_j-B_j$. Denote it by $i_0$.
 After renumeration of representations $\pi_i$, we can assume that $i_0=1$. Therefore, we can assume that the following  holds
 
 \begin{enumerate}
  \item[(i)] $C_1+D_1\geq C_i+D_i,\quad  2\leq i\leq k;$

\item[(ii)] if $C_1+D_1= C_i+D_i$ for some $i\in\{1,2,\dots,k\}$, then 
%either $D_1=B_1$ or 
$ D_1-B_1\leq D_i-B_i$.
 \end{enumerate}
 We shall now complete the proof of the lemma. First consider the case $B_1=D_1$. By the inductive assumption, we have epimorphism
 $$
 \lambda(d_2+\dots+d_k) \twoheadrightarrow \pi_2\t\dots\t\pi_k.
 $$
This implies that we have an epimorphism
$$
 \pi_1\t\lambda(d_2+\dots+d_k) \twoheadrightarrow  \pi_1\t\pi_2\t\dots\t\pi_k.
 $$
The fact that $B_1=D_1$ and (i) imply that 
$$
\pi_1\t\lambda(d_2+\dots+d_k)\cong \lambda(d_1+d_2+\dots+d_k).
$$
Therefore, (2) holds in this case.

It remains to consider the case
$$
1\leq D_1-B_1.
$$
 Denote 
 $$
 C_1'=C_1-1, \qquad D_1'=D_1-1.
 $$
 and
 \begin{equation*}
 \pi_1'=
 u_{ess}
\left(
\begin{matrix}
A_1  \hskip3mm & B_1 \hskip3mm
\\
    \hskip3mm C_1' &    \hskip3mm D_1'
\end{matrix}
 \right)^{(\rho)}.
 \end{equation*}
 Denote by $d_1'$ the multi set which satisfies 
 $$
 \pi_1'=L(d_1').
$$ 
 Now by previous lemma,
 $$
 \pi_1'\t \pi_i
 $$
 is irreducible for all $2\leq i\leq k$. Denote
 $$
 \d_1=\d([C_1,D_1]^{(\rho)}).
 $$
 
 Now we have an epimorphism
 $$
\lambda(d_1)\cong \d_1\t\lambda(d_1') \twoheadrightarrow \d_1\t\pi_1'.
 $$
 Therefore, we have an epimorphism
 $$
 \d_1\t\pi_1' \twoheadrightarrow \pi_1,
 $$
 since $\lambda(d_1)$ has a unique irreducible quotient, and that quotient is $\pi_1$.
 Further, we have an epimorphism
 \begin{equation}
 \label{pomocno2}
 \d_1\t\pi_1' \t\pi_2\t\dots\t\pi_k \twoheadrightarrow \pi_1\t\pi_2\t\dots\t\pi_k.
 \end{equation}

 By the inductive assumption we have an epimorphism
 $$
 \lambda(d_1'+d_2+\dots+d_k) \twoheadrightarrow \pi_1'\t\pi_2\t\dots\t\pi_k.
 $$
 Therefore, we have an epimorphism
 $$
 \d_1\t\lambda(d_1'+d_2+\dots+d_k) \twoheadrightarrow \d_1\t\pi_1'\t\pi_2\t\dots\t\pi_k.
 $$
 Observe that by our choice in (i), 
 $$
 \d_1\t\lambda(d_1'+d_2+\dots+d_k)\cong \lambda(d_1+d_2+\dots+d_k).
 $$
 Then the last two relations and \eqref{pomocno2} imply the claim (2) from the lemma.
 The proof is now complete.
\end{proof}

In the same way as A. Zelevinsky proved Proposition 8.5 from \cite{Z} (looking also the contragredient setting),  follows the next theorem   from previous lemma\footnote{Proposition in I.9 of \cite{MoeW-GL} is  closely related to this theorem (in the field case)}.

\begin{theorem} Suppose that we have essentially Speh representations $\pi_1,\dots,\pi_k$. Then 
$$
\pi_1\t\dots\t\pi_k
$$
is irreducible if and only if  the representations
$$
\pi_i\t\pi_j
$$
are irreducible for all $1\leq i<j\leq l.$
\qed
\end{theorem}

Let $P$ be a parabolic subgroup  of $GL(n,\C A)$ with a Levi decomposition $P=MN$,  let $\pi$ be an irreducible unitary representation of $M$ and let $\varphi$ be a (not necessarily unitary) character of $M$. Then Theorem 2.3 implies that the parabolically induced representation
$$
\text{Ind}_P^{GL(n,\C A)}(\varphi\pi)
$$
is equivalent to a representation $\pi_1\t\dots\t\pi_k$ considered in the above theorem.
Therefore, the above theorem gives an explicit necessary and sufficient condition for the  representation of type $\text{Ind}_P^{GL(n,\C A)}(\varphi\pi)$ to be irreducible.

\section{Relation with a result of C. M\oe glin and J.-L. Waldspurger}

In this section we shall recall of the   (sufficient) irreducibility criterion for representations \eqref{repr}  in the case of Speh representations,   obtained in Lemma I.6.3 of \cite{MoeW-GL}.
We follow the notation of \cite{MoeW-GL}, and  assume in this section that $\C A$ is a  field  (non-commutative division algebras are not considered in  \cite{MoeW-GL}). 

\subsection{Some notation}
Let $\d\in \DS^u$.
As in \cite{MoeW-GL}, write
$$
\d[s]=\nu^s\d.
$$
Consider the following two parameters attached to $\d$, the (unitarizable) cuspidal representation $\rho$ and $t\in (1/2)\Z_{\geq 0}$:
$$
\d=
\d([-t,t]^{(\rho)}).
$$
Let $a,b\in\R$ such that $b-a\in \Z_{\geq 0}$. Let 
$$
J(\d,a,b)=L(\nu^a\d, \nu^{a+1}\d,\dots,\nu^{b}\d)=L(\d[a], \d[a+1],\dots,\d[b])
$$
$$
=L(\d([-t,t]^{(\rho)})[a], \d([-t,t]^{(\rho)})[a+1],\dots,\d([-t,t]^{(\rho)})[b])
$$
$$
=
L(\d([a-t,a+t]^{(\rho)}), \d([a+1-t,a+1+t]^{(\rho)}),\dots,\d([b-t,b+t]^{(\rho)})).
$$
Therefore for $\d=\d(\rho,2t+1)=\d([-t,t]^{(\rho)})$ we have
\begin{equation}
\label{connection}
J(\d,a,b)=
u_{ess}
\left(
\begin{matrix}
a-t  \hskip3mm & a+t  \hskip3mm
\\
    \hskip3mm b-t&    \hskip3mm b+t
\end{matrix}
\right)
^{(\rho)}.
 \end{equation}
 
 \subsection{Linking condition of \cite{MoeW-GL}} 
  Take two essentially Speh representations  $J(([-t,t]^{(\rho)}),a,b)$ \, and  \, $J(\d([-t',t']^{(\rho')}),a',b')$ as above. Then they are called linked if
 
 \begin{enumerate}
 
 \item $\rho\cong \rho'$;
 
 \item $(a-t)-(a'-t')\in\Z$;
 
 \item $b>b'+|t-t'|, $ \ \ \  $a>a'+|t-t'| $ \ and \ $a-b'\leq 1+t+t'$, 
 \qquad
 or 
 \\
 $b'>b+|t-t'|, \ \ \ $ $a'>a+|t-t'| $ \ and \ $a'-b\leq 1+t+t'$.
 
 \end{enumerate}
 
\subsection{Irreducibility result of C. M\oe glin and J.-L. Waldspurger} Now we recall of a claim in (ii) of Lemma in I.6.3 of \cite{MoeW-GL}: if  $J(([-t,t]^{(\rho)}),a,b)$ \, and  \, $J(\d([-t',t']^{(\rho')}),a',b')$ are not linked, then 
 $$
 J(([-t,t]^{(\rho)}),a,b)\t J(\d([-t',t']^{(\rho')}),a',b')
 $$
  is irreducible.

\subsection{Another interpretation of the linking condition} We  analyze the first condition in (3).   This condition  is equivalent to the fact  that the following hold:

$b-t>b'-t', $ \ \ \  $a-t>a'-t' $ \ and \ $a-t\leq b'+t'+1$

and

$b+t>b'+t', $ \ \ \  $a+t>a'+t' $ \ and \ $a-t\leq b'+t'+1$.

Write
$$
\left(
\begin{matrix}
A  \hskip3mm & B  \hskip3mm
\\
    \hskip3mm C&    \hskip3mm D
\end{matrix}
 \right)
: =
\left(
\begin{matrix}
a-t  \hskip3mm & a+t  \hskip3mm
\\
    \hskip3mm b-t&    \hskip3mm b+t
\end{matrix}
\right),
$$
$$
\left(
\begin{matrix}
A ' \hskip3mm & B'  \hskip3mm
\\
    \hskip3mm C'&    \hskip3mm D'
\end{matrix}
 \right)
: =
\left(
\begin{matrix}
a'-t'  \hskip3mm & a'+t'  \hskip3mm
\\
    \hskip3mm b'-t'&    \hskip3mm b'+t'
\end{matrix}
\right).
$$
Now the first condition in (3) becomes

$C'<C, $ \ \ \  $A'<A $ \ and \ $A\leq D'+1$

and

$D'<D, $ \ \ \  $B'<B $ \ and \ $A\leq D'+1$,

i.e.
$$
\left(
\begin{matrix}
A ' \hskip3mm & B'  \hskip3mm
\\
    \hskip3mm C'&    \hskip3mm D'
\end{matrix}
 \right)
 <{strong}
 \left(
\begin{matrix}
A  \hskip3mm & B  \hskip3mm
\\
    \hskip3mm C&    \hskip3mm D
\end{matrix}
 \right)
 $$
 and $[A',D']_\Z\cup [A,D]_\Z$ is a $\Z$-segment.

Analysis of the second condition in (3) gives the opposite inequality.

Therefore, the above analysis of the linking condition together with the formula \eqref{connection}, tells us that  the irreducibility result of C. M\oe glin and J.-L. Waldspurger from Lemma I.6.3 (recalled in 8.3) is exactly one implication of the irreducibility criterion of Theorem 7.2 in the case that $\C A$ is a field.

\section{Relation with a result of B. Leclerc, M. Nazarov and J.-Y. Thibon}

In this section we describe what gives  specialization of  Theorem 1 from \cite{LNT} to the case of unramified representations
of general linear groups over a non-archimedean local field $F$.

Let $\alpha=(\alpha_1,\dots,\alpha_r)$, $\alpha_1\geq \alpha_2\geq\dots\geq \alpha_r$ be a partition of $m\geq1$ into positive integers, and $x\in \mathbb Z$. Denote by 
$$
\mathbf m(\alpha,x)=\sum_{i=1}^r\ ([x-i+1,x-i+\alpha_i]) \in M(\C S(\Z)).
$$
To such multiset $\mathbf m(\alpha,x)$  attach the set
$$
\mathcal I(\alpha,x) \ = \ \langle-\infty,x-r]\ \ \cup\ \ \underset{i=1}{\overset{r}\cup}\ \{x-i+\alpha_i+1\}.
$$
In other words, $\mathcal I(\alpha,x)$ consist of all the integers which are strictly left from the support of $\mathbf m(\alpha,x)$, together with all the integers that one gets increasing each end of segment in $\mathbf m(\alpha,x)$ by 1. Let $\rho\in \C C$ and denote $\pi=L(\mathbf m(\alpha,x)^{(\rho)}).$ Then we define
$$
\C I_\rho(\pi):=\mathcal I(\alpha,x).
$$
Observe that above $\pi$  is a special case of a ladder representation defined in \cite{LMi}.

We  can  graphically interpret $\mathbf m(\alpha,x)$ by 
\begin{equation} 
\label{g1111111}
\xymatrix@C=.6pc@R=.1pc
{ 
x &&&&   \ \bullet& \bullet\ar @{-}[l] & \bullet\ar @{-}[l] & \bullet\ar @{-}[l] & \bullet\ar @{-}[l] & \bullet\ar @{-}[l] & \bullet\ar @{-}[l] & \bullet\ar @{-}[l] &  x-1+\alpha_1
\\
x-1&&&  \ \bullet& \bullet\ar @{-}[l] & \bullet\ar @{-}[l] & \bullet\ar @{-}[l] & \bullet\ar @{-}[l] & \bullet\ar @{-}[l] & \bullet\ar @{-}[l]  &  && x-2+\alpha_2
\\ 
&&&..&..&..&..&..&&&
\\
x-(r-1)+1&&  \ \bullet& \bullet\ar @{-}[l] & \bullet\ar @{-}[l] & \bullet\ar @{-}[l] & \bullet\ar @{-}[l] & \bullet\ar @{-}[l] &   &&&& x-r+\alpha_r
\\
x-r+1&   \bullet& \bullet\ar @{-}[l] & \bullet\ar @{-}[l] & \bullet\ar @{-}[l] & \bullet\ar @{-}[l] &  &   &&&&& x-r+\alpha_r,
} 
\end{equation}
and further, $\mathcal I(\alpha,x)$  can be graphically interpreted by
\begin{equation} 
\label{g22}
\xymatrix@C=.6pc@R=.1pc
{ 
\cdot\cdot\cdot & \circ & \circ & \circ & \circ \hskip16mm  & &&&&&  \circ && \hskip3mm \circ &..&\hskip1mm\circ&& \hskip1mm \circ \hskip20mm
} 
\end{equation}

For a subset $X$ of $\mathbb Z$ we  denote by $[X]_\mathbb Z$ the smallest segment in $\mathbb Z$ containing $X$.

Now we have

\begin{theorem} $($\cite{LNT}$)$
 \label{LNT-Th} 
 Let $\alpha=(\alpha_1,\dots,\alpha_r)$, $\alpha_1\geq \alpha_2\geq\dots\geq \alpha_r$ and $\beta=(\beta_1,\dots,\beta_s)$, $\beta_1\geq \beta_2\geq\dots\geq\beta_s$ be partitions of positive integers $m$ and $n$ respectively, and $x,y\in \mathbb Z$. Suppose that $\chi$ is an unramified character of $F^\times$. Then:
\begin{enumerate}

\item If
$$
y<x,
$$
then $L(\mathbf m(\alpha,x)^{(\chi)})\times L(\mathbf m(\beta,y)^{(\chi)})$ reduces $\iff$\footnote{In the paper \cite{LNT} is the condition: if there exist $i<j<k$ such that $i,k\in \mathcal I(\alpha,x)\backslash \mathcal I(\beta,y)$ and $j \in \mathcal I(\beta,y)\backslash \mathcal I(\alpha,x)$. }
$$
[\mathcal I(\alpha,x)\backslash \mathcal I(\beta,y)]_\mathbb Z \cap (\mathcal I(\beta,y)\backslash \mathcal I(\alpha,x))\ne\emptyset.
$$

\item The case $$x<y$$ reduces to the previous case applying commutativity of $R$, and the condition for reducibility becomes $\iff$
$$
(\mathcal I(\alpha,x)\backslash \mathcal I(\beta,y)) \cap [\mathcal I(\beta,y)\backslash \mathcal I(\alpha,x)]_\mathbb Z\ne\emptyset.
$$

\item Suppose
$$
x=y.
$$
Then $L(\mathbf m(\alpha,x)^{(\chi)})\times L(\mathbf m(\beta,y)^{(\chi)})$ reduces $\iff$\footnote{In the paper \cite{LNT} is the condition: if there exist $i,j,k,l$ such that $i,k\in \mathcal I(\alpha,x)\backslash \mathcal I(\beta,y)$ and $j,l \in \mathcal I(\beta,y)\backslash \mathcal I(\alpha,x)$ satisfying either $i<j<k<l$ or $j<i<l<k.$ }
$$
[\mathcal I(\alpha,x)\backslash \mathcal I(\beta,y)]_\mathbb Z \cap (\mathcal I(\beta,y)\backslash \mathcal I(\alpha,x))\ne\emptyset
$$
and
$$
(\mathcal I(\alpha,x)\backslash \mathcal I(\beta,y)) \cap [\mathcal I(\beta,y)\backslash \mathcal I(\alpha,x)]_\mathbb Z\ne\emptyset.
$$

\end{enumerate}
\end{theorem}

Consider   representations
$$
\pi_i=u_{ess}
\left(
\begin{matrix}
A_i  \hskip3mm & B_i  \hskip3mm
\\
    \hskip3mm C_i&    \hskip3mm D_i
\end{matrix}
\right)^{(\chi)}, \ i=1,2.
$$
We shall now apply Theorem \ref{LNT-Th} to test reducibility of $\pi_1\t\pi_2$.

Observe that
$$
\mathcal I_\chi(\pi_i)=\langle-\infty,A_i-1] \cup [B_i+1,D_i+1].
$$
We consider three cases. 

Suppose $C_1=C_2$. Without lost of generality we can assume $A_1\leq A_2$. Then $\mathcal I_\chi (\pi_1)\backslash \mathcal I_\chi (\pi_2)$ $= [B_1+1,D_1+1]\backslash [B_2+1,D_2+1],$ and therefore 
$$
[\mathcal I_\chi (\pi_1)\backslash \mathcal I_\chi (\pi_2)]_\mathbb Z = [B_1+1,D_1+1]\backslash [B_2+1,D_2+1].
$$
From the other side
$$
\mathcal I_\chi (\pi_2)\backslash \mathcal I_\chi (\pi_1)\subseteq \Z\backslash [B_1+1,D_1+1].
$$
These two subsets are disjoint. Therefore, Theorem \ref{LNT-Th} implies irreducibility.

Suppose now $C_1<C_2$. We first consider the case $A_2\leq A_1$. Now $\mathcal I_\chi (\pi_2)\backslash \mathcal I_\chi (\pi_1)= [B_2+1,D_2+1]\backslash [B_1+1,D_1+1],$ and therefore 
$$
[\mathcal I_\chi (\pi_2)\backslash \mathcal I_\chi (\pi_1)]_\mathbb Z = [B_2+1,D_2+1]\backslash [B_1+1,D_1+1].
$$
This is obviously disjoint with $\mathcal I_\chi (\pi_1)\backslash \mathcal I_\chi (\pi_2)$.  Now Theorem \ref{LNT-Th} implies irreducibility.

It remains to consider the case $C_1<C_2$ and $A_1< A_2$. Suppose first that $D_2\leq D_1$. Then 
$$
\mathcal I_\chi (\pi_1)\backslash \mathcal I_\chi (\pi_2)= [B_1+1,D_1+1]\backslash [B_2+1,D_2+1]  .
$$
 Further
$$
\mathcal I_\chi (\pi_2)\backslash \mathcal I_\chi (\pi_1)= [A_1, A_2-1]\cup [B_2+1,D_2+1]\backslash [B_1+1,D_1+1].
$$
Now $D_2\leq D_1$ implies
$$
[\mathcal I_\chi (\pi_2)\backslash \mathcal I_\chi (\pi_1)]_\Z \subseteq [A_1, B_1].
$$
We see again disjointness and again get irreducibility.

We are left with the case $C_1<C_2,A_1< A_2$ and $D_1 < D_2$. Recall
$$
\mathcal I_\chi (\pi_1)\backslash \mathcal I_\chi (\pi_2)= [B_1+1,D_1+1]\backslash [B_2+1,D_2+1]  
$$
and 
$$
\mathcal I_\chi (\pi_2)\backslash \mathcal I_\chi (\pi_1)= [A_1, A_2-1]\cup [B_2+1,D_2+1]\backslash [B_1+1,D_1+1].
$$
Suppose $B_2\leq B_1$. Then $\mathcal I_\chi (\pi_1)\backslash \mathcal I_\chi (\pi_2)$ is empty, which implies irreducibility.
Suppose now $B_1\leq B_2$. Then $[\mathcal I_\chi (\pi_2)\backslash \mathcal I_\chi (\pi_1)]_\Z= [A_1,D_2+1]$. $B_1\in \mathcal I_\chi (\pi_2)\backslash \mathcal I_\chi (\pi_1)$. Now obviously, $B_1$ is in the intersection. This implies reducibility. 

Therefore, we have just seen that Theorem \ref{LNT-Th} implies our result in the unramified  case when $\C A $ is a field.

\begin{remark}
Observe that the theory of types for general linear groups over division algebras developed in \cite{Se1} - \cite{SeSt4}, together with the theory of covers from \cite{BK}, should relatively easy imply that Theorem \ref{LNT-Th} holds if one puts any $\rho\in \mathcal C$ instead of $\chi$ (see \cite{Se} and \cite{BHLS} for such applications of \cite{Se1} - \cite{SeSt4} and \cite{BK}; we have not checked all details for the implication  in the case that we consider in this section). Therefore, the main result of our paper should follow from \cite{LNT} using \cite{Se1} - \cite{SeSt4} and \cite{BK}.
\end{remark}


\begin{thebibliography}{99}


\bibitem{Au}
 Aubert, A. M., {\it Dualit\'{e} dans le groupe de Grothendieck de la cat\'{e}gorie
des repr\'{e}sentations lisses de longueur finie d'un groupe r\'{e}ductif $p$-
adique}, Trans. Amer. Math. Soc.  347 (1995), 2179-2189;  {\it Erratum},
Trans. Amer. Math. Soc  348 (1996),  4687-4690


 
 \bibitem{Ba-Sp}
 Badulescu, A.~I., {\it On p-adic Speh representations}, Bulletin de la SMF, to appear.
 
\bibitem{BHLS}  Badulescu, A. I., Henniart, G., Lemaire, B. and S\'echerre, V., {\it Sur le dual unitaire de $GL_r(D)$},  Amer. J. Math. 132 no. 5 (2010),  1365-1396.


\bibitem{BLM}  Badulescu, A.~I., Lapid, E. and M\'ingues, A., {\it Une condition suffisante pour l'irreducibilite d'une induite parabolique de $GL(m,D)$}, Ann. Inst. Fourier, to appear.







\bibitem{BR-Tad}
 Badulescu, A.~I. and Renard, D.~A.,
{\it Sur une conjecture de Tadi\'c},
 Glasnik Mat.
  39 no. 1
(2004),
 49-54.
 
 \bibitem{BR-inv}
 Badulescu, A.~I. and Renard, D.~A., {\it Zelevinsky involution and M\oe glin-Waldspurger algorithm for $GL_n(D)$}, 
Functional analysis IX, 9--15, Various Publ. Ser. (Aarhus), 48, Univ. Aarhus, Aarhus, 2007.

\bibitem{BR-arch}
 Badulescu, A.~I. and Renard, D.~A.,
{\it
Unitary dual of $GL_n$ at archimedean places and global Jacquet-Langlands correspondence},
 Compositio Math. 146 no. 5 (2010),  1115-1164.


\bibitem{BK}
 Bushnell, C. J. and Kutzko, P. C., {\it Smooth representations of reductive $p$-adic groups: structure theory via types}, Proc. London Math. Soc. (3) 77 (1998), no. 3, 582-634.



\bibitem{DKV}
 Deligne, P.,  Kazhdan, D. and Vign\'eras, M.-F.,
{\it Repr\'esentations des alg\`ebres centrales simples 
$p$-adiques}, in book
"Repr\'esentations des Groupes
R\'eductifs sur un Corps Local" by Bernstein, J.-N.,
Deligne, P.,  Kazhdan, D. and Vign\'eras, M.-F. ,
Hermann, Paris, 1984.

 \bibitem{LMi} Lapid, E. and M\'inguez, A.,
{\it On a determinantal formula of Tadi\'c}, {\it Amer. J.  Math.}, to appear.




\bibitem{LNT} Leclerc, B., Nazarov, M. and Thibon, J.-Y., {\it Induced representations of affine Hecke algebras and canonical bases of quantum groups}, Studies in memory of Issai Schur (Chevaleret/Rehovot, 2000), 115-153, Progr. Math., 210, Birkh\"user Boston, Boston, MA, 2003.


\bibitem{L}
Lusztig, G.,
{\it Quivers, perverse sheaves, and quantized enveloping algebras},
J. Amer. Math. Soc. 4 (1991), no. 2, 365-421. 

\bibitem{M}
M\'inguez, A., {\it Sur l'irr\'eductibilit\'e d'une induite parabolique}, J. Reine Angew. Math. 629 (2009), 107-131.

\bibitem{MS}
M\'inguez, A. and S\'echerre, V., {\it Repr\'esentations banales de GL(m,D)}, Compos. Math., to appear.

 \bibitem{MoeW-alg}
 M{\oe}glin, C. and Waldspurger, J.-L., {\it Sur l'involution de Zelevinski}, J.  Reine Angew. Math. 372 (1986), 136-177. 



 \bibitem{MoeW-GL}
 M{\oe}glin, C. and Waldspurger, J.-L.,
{\it  Le spectre residuel de $GL(n)$},
 Ann. Sci. \'{E}cole Norm. Sup 
 22
 (1989),
 605-674.
 

 \bibitem{Ro}
Rodier, F.
{\it  Repr\'esentations de $GL(n,k)$ o\`u $k$ est un
corps $p$-adique},
 S\'eminaire Bourbaki no 587 (1982), Ast\'erisque
 92-93 (1982),
 201-218.


 \bibitem{ScSt}  Schneider, P. and  Stuhler, U., 
{\it Representation theory and sheaves on the Bruhat-Tits building},
Publ. Math. IHES 85
(1997),
 97-191.
 
\bibitem{Se1}  S\'echerre, V., {\it Repr\'esentations lisses de $GL(m,D)$. I. Caract\`eres simples}, Bull. Soc. Math. France 132 (2004), no. 3, 327-396.
  
\bibitem{Se2}   S\'echerre, V., {\it Repr\'esentations lisses de $GL(m,D)$. II. $\beta$-extensions}, Compos. Math. 141 (2005), no. 6, 1531-1550.
   
\bibitem{Se3}   S\'echerre, V., {\it Repr\'esentations lisses de $GL_m(D)$. III. Types simples}, Ann. Sci. \'Ecole Norm. Sup. (4) 38 (2005), no. 6, 951-977.
   
\bibitem{SeSt4}    S\'echerre, V. and Stevens, S., {\it Repr\'esentations lisses de $GL_m(D)$. IV. Repr\'esentations supercuspidales},  J. Inst. Math. Jussieu 7 (2008), no. 3, 527-574.

\bibitem{Se}
 S\'echerre, V.,
{\it Proof of the Tadi\'c conjecture (U0) on the unitary dual
of GL$_m(D)$},
 J. reine angew. Math.
 626 (2009),
 187-203.
 
 \bibitem{Sp} Speh, B., {\em Unitary representations of $GL(n, \mathbb R)$ with non-trivial
$(g, K)$- cohomology}, Invent. Math. { 71} (1983), 443-465.

\bibitem{T-R-C-old}
Tadi\'c, M., {\em
Unitary  representations of general linear group over  real and complex field, }
preprint MPI/SFB 85-22 Bonn (1985).

\bibitem{T-AENS}
Tadi\'c, M., {\em Classification of unitary representations in irreducible
representations of general linear group (non-archimedean case)}, Ann. Sci.
Ecole Norm. Sup. { 19} (1986), 335-382.




\bibitem{T-div-a}
Tadi\'{c}, M.,
{\it Induced representations of $GL(n,A)$ for
$p$-adic division algebras $A$},
 J. Reine Angew. Math.
 405
 (1990),
 48-77.
 
 
 \bibitem{T-irr}
Tadi\'c, M., {\em
On reducibility of parabolic induction,}  Israel J. Math.  { 107}  (1998), 29-91.



\bibitem{T-R-C-new}
Tadi\'c, M., {\em $GL(n,\mathbb C)\hat{\ }$ and
$G L(n,\mathbb R)\hat{\ }$},
in ``Automorphic Forms and $L$-functions II,
Local Aspects",
Contemporary Mathematics
489,
2009,
 285-313.

\bibitem{T-CJM} Tadi\'c, M.,
{\it On reducibility and
unitarizability for classical $p$-adic groups, some general results}, 
 Canad. J. Math. 61 (2009),  427-450.
 


\bibitem{Z} Zelevinsky, A. V.,
{\it  Induced representations of reductive p-adic groups II. On
irreducible representations of GL(n)},
    Ann. Sci. \'{E}cole Norm. Sup.
  13 (1980),
pp. 165-210.






\end{thebibliography}
\end{document}